\documentclass[12pt]{amsart}
\usepackage{geometry}   

\usepackage[colorlinks,citecolor = red, linkcolor=blue,hyperindex]{hyperref}
\usepackage{euscript,eufrak,verbatim, mathrsfs}
\usepackage[psamsfonts]{amssymb}
\usepackage{bbm}
\usepackage{graphicx}
\usepackage{float}
\usepackage{float, tikz}
\usepackage{subcaption}
\usetikzlibrary{calc}
\usetikzlibrary{arrows.meta,patterns}
\usepackage{pgfplots}
\usepgfplotslibrary{fillbetween}
\pgfplotsset{compat=1.18}

\usepackage{extpfeil}

\usepackage[all, cmtip]{xy}

\usepackage{upref, xcolor, dsfont}
\usepackage{amsfonts,amsmath,amstext,amsbsy, amsopn,amsthm}
\usepackage{enumerate}

\usepackage{url}

\usepackage{mathtools}
\usepackage{bookmark}

\usepackage{euscript}
\usepackage{times}
\usepackage{type1cm}        
\usepackage{multicol}        
\usepackage[bottom]{footmisc}
\newcommand{\normmm}[1]{{\left\vert\kern-0.25ex\left\vert\kern-0.25ex\left\vert #1 
		\right\vert\kern-0.25ex\right\vert\kern-0.25ex\right\vert}}

\newtheorem{theorem}{Theorem}[section]
\newtheorem*{theorem*}{Theorem B}
\newtheorem{lemma}[theorem]{Lemma}
\newtheorem{convention}[theorem]{Convention}

\newtheorem{proposition}[theorem]{Proposition}
\newtheorem{properties}[theorem]{Elementary Properties}
\newtheorem{corollary}[theorem]{Corollary}

\newtheorem*{observation*}{Observation}

\newtheorem*{assumption*}{Assumption}
\newtheorem*{question*}{Question}

\theoremstyle{definition}

\newtheorem*{definition*}{Definition}

\theoremstyle{remark}

\newtheorem*{remark*}{Remark}

\newcommand{\R}{\mathbb{R}}
\newcommand{\N}{\mathbb{N}}

\newcommand{\E}{\mathbb{E}}

\newcommand{\PP}{\mathbb{P}}

\newcommand{\Var}{\mathrm{Var}}

\newcommand{\supp}{\mathrm{supp}}

\newcommand{\odd}{\mathrm{odd}}
\newcommand{\even}{\mathrm{even}}
\newcommand{\dist}{\mathrm{dist}}

\newcommand{\GMC}{\mathrm{GMC}}

\newcommand{\an}{\text{\, and \,}}

\numberwithin{equation}{section}

\begin{document}
	\title[Fourier dimension of GMC]{Harmonic analysis of multiplicative chaos  \\ Part I:  the proof of Garban-Vargas conjecture for 1D GMC}

	\author
	{Zhaofeng Lin}
	\address
	{Zhaofeng LIN: School of Fundamental Physics and Mathematical Sciences, HIAS, University of Chinese Academy of Sciences, Hangzhou 310024, China}
	\email{linzhaofeng@ucas.ac.cn}

	\author
	{Yanqi Qiu}
	\address
	{Yanqi QIU: School of Fundamental Physics and Mathematical Sciences, HIAS, University of Chinese Academy of Sciences, Hangzhou 310024, China}
	\email{yanqi.qiu@hotmail.com, yanqiqiu@ucas.ac.cn}

	\author
	{Mingjie Tan}
	\address
	{Mingjie TAN: School of Mathematics and Statistics, Wuhan University, Wuhan 430072, China}
	\email{mingjie1tan1wuhan@gmail.com}


	\begin{abstract}
	  In this paper,  we establish the exact Fourier dimensions of all standard sub-critical Gaussian multiplicative chaos on the unit interval, thereby confirming the Garban-Vargas conjecture. The proof relies on a significant improvement of the vector-valued martingale method, initially developed by Chen-Han-Qiu-Wang in the studies of the Fourier dimensions of Mandelbrot cascade random measures.  
	\end{abstract}


	\maketitle

	\setcounter{tocdepth}{2}

	\setcounter{tocdepth}{0}
	\setcounter{equation}{0}



	\section{Introduction}\label{sec-intro}

	This paper is the first part of a series of works on the harmonic analysis of multiplicative chaos measures in the Euclidean spaces.  This series will provide a systematic development of the {\it vector-valued martingale method}, discovered in \cite{CHQW24}, for analyzing the polynomial Fourier decay of multiplicative chaos measures. This method seems to be fundamental,  powerful and  straightforward,   yielding  a crucial {\it random Fourier decoupling estimate} (see Proposition \ref{dec-vm-bis} below),  which can be naturally integrated with classical {\it Littlewood-Paley type decomposition}.  As a result, for key models in multiplicative chaos theory, we are able to determine the exact Fourier dimensions (i.e., the optimal exponent of polynomial Fourier decay) of the associated multiplicative chaos measures.
	
	The  main topics of the three parts in this series  are as follows: 
	\begin{itemize}
	\item {\bf Part I.}  The case study of  the key model--the standard log-correlated Gaussian multiplicative chaos on the unit interval (see \S \ref{sec-GMC-back} for a brief introduction). We  prove the Garban-Vargas conjecture for the standard sub-critical 1D GMC.  
	\item {\bf Part II.} A unified theory of Fourier decay of multiplicative chaos on Euclidean spaces. In this part,  we develop an axiomatic theory, allowing us to deal with various classical subcritical multiplicative chaos models,  including 2D GMC and  higher dimensional GMC,  Mandelbrot random covering, Poisson multiplicative chaos, canonical or generalized Mandelbrot cascades and beyond. 
	\item {\bf Part III.}  The general theory of Fourier decay of multiplicative chaos in the more abstract setting and with abstract background measures. In particular,  in this part, we will deal with complex GMC and various multiplicative chaos on certain natural fractal sets with positive Fourier dimensions. 
	\end{itemize}

An informal description of our main task in Part I is to find the optimal exponent $\tau \ge 0$ such that  (all the notation will be introduced later) the following inequality 
	\begin{align}\label{vec-meth}
	\mathbb{E}\Big[\Big\{\sum_{n=1}^{\infty}\big|n^{\frac{\tau}{2}} \widehat{\mu_{\gamma,\mathrm{GMC}}}(n) \big|^{q}\Big\}^{\frac{1+\varepsilon}{q}}\Big]<\infty
		\end{align}
	holds for very large $q>2$ and very small $\varepsilon>0$, where $\mu_{\gamma,\mathrm{GMC}}$ is the GMC measure on the unit interval associated to the parameter $\gamma\in(0,\sqrt{2})$. Equivalently, we shall find explicitly the following  critical quantity
	\[
 	\sup\Big\{\tau\ge 0\,\big| \,\text{there exists a quantity $(\tau, q, \varepsilon)$ with $q>2$ and $\varepsilon>0$ such that \eqref{vec-meth} holds}\Big\}
	\] 
	and will prove that this critical quantity is indeed the desired Fourier dimension. 
	
Both Part I and Part II are self-contained and can thus be read independently.  The main result of Part I should be considered as  a special case of that of Part II.  However,   we believe that it is reasonable to  write a self-contained separate paper on 1D Gaussian multiplicative chaos, for at least the following reasons: 1) The study of this model is the original motivation of  the general theory. 2)  The Part I will provide readers with immediate insight into our work's novel ideas and methods.  Indeed, the analysis and main inequalities for this model are more concrete and  simpler,  the notation is also much easier. 3) There will be extra difficulties in the study of the higher dimensional  GMC, for instance, for the planar 2D GMC, we shall use a new method by constructing a special $*$-scale invariant kernel (which will be called  {\it sharply-$\sigma$-regular} there) for the log-correlated Gaussian field.
	  
	The common framework of this series is Kahane's $T$-martingale theory for multiplicative chaos measures \cite{Kah87}. More precisely, we will consider the Euclidean space $\R^d$ or a certain subset  $U\subset \R^d$ (for instance,  $U = [0,1]^d$) equipped with a finite Radon measure $\nu$ and  a sequence of {\it independent random  non-negative functions} $P_n(t)$ with $t\in \R^d$ such that $\E[P_n(t)] \equiv 1$.  For any $n\ge 1$, define a random measure by 
	\[
	\mu_n (\mathrm{d}t) : = \Big[\prod_{k=1}^n P_k(t) \Big] \nu(\mathrm{d}t).
	\]
	Then, Kahane's general $T$-martingale theory asserts that, almost surely,   the sequence of random measures $\mu_n$  converges weakly to a random measure  \[
	\mu_\infty = \lim_{n\to\infty}\mu_n.
	\] This limiting random measure $\mu_\infty$ will be  referred to as the multiplicative chaos measure associated with the random sequence of functions $(P_n)_{n\ge 1}$ and the background measure $\nu$.

	\subsection{Background on GMC}\label{sec-GMC-back}
	 Gaussian multiplicative chaos (GMC),  introduced by Kahane in the 1980s \cite{Kah85a}, is a theory of  random measures. Informally, GMC measures arise as the exponential of log-correlated Gaussian fields. They are closely related to many models in mathematical physics such as 2D quantum gravity \cite{DS11, DKRV16, KRV19}, SLE \cite{AJKS11, She16} and random matrices \cite{Web15, BWW17}. The reader is refered to \cite{RV14} for a recent review on GMC.

	Recently, the Fourier decay and the Fourier dimension of GMC  have gained considerable attention. Falconer and Jin \cite{FJ19} provided a non-trivial lower bound for the Fourier dimension of 2D GMC for small parameter values $\gamma < \frac{1}{33}\sqrt{858 - 132\sqrt{34}}$.   The Fourier coefficients of GMC have also been studied  in the construction of the Virasoro algebra in Liouville conformal field theory, number theory and random matrix theory, see \cite{BGKRV24, CN19} and their references.  In a remarkable recent work \cite{GV23},  Garban and Vargas established the Rajchman property of the standard sub-critical GMC measure (denoted $M_\gamma$ there) on the unit circle. Namely, for all sub-critical parameters $\gamma\in(0,\sqrt{2})$,  they established the almost sure convergence 
	$
	\lim_{n\to\infty} \widehat{M_\gamma}(n) = 0$. 
Moreover, for small parameters $0<\gamma<1/\sqrt{2}$, 	they also obtain a lower bound of the Fourier dimension (the definition of Fourier dimension of a measure on $[0,1]$ is recalled in \S~\ref{sec-fd}):
	\[
	\frac{1}{2} - \gamma^2 \leq \dim_F(M_\gamma) \leq 1 - \gamma^2 < \dim_H(M_{\gamma}) = 1 - \frac{\gamma^2}{2},
	\]
	where $\dim_F(M_\gamma)$ and $\dim_H(M_\gamma)$ denote the Fourier and Hausdorff dimensions of $M_\gamma$ respectively.  	Based on the rescaled fluctuation of $\widehat{M}_\gamma(n)$ (see \cite[Theorem~1.3]{GV23}), Garban and Vargas conjectured that for small parameters $0<\gamma < 1/\sqrt{2}$, the Fourier dimension of $M_\gamma$ is given by $1 - \gamma^2$. Moreover, they asked whether for all the sub-critical parameters $\gamma\in(0,\sqrt{2})$,  the Fourier dimension of $M_\gamma$ coincides with its correlation dimension. In this paper, we resolve the Garban-Vargas conjecture in the affirmative.

	\subsection{Main result}
Consider the sub-critical  GMC measure $\mu_{\gamma,\mathrm{GMC}}$ with $\gamma\in(0,\sqrt{2})$ on the unit interval. Informally, $\mu_{\gamma,\mathrm{GMC}}$  is  a random measure on $[0,1]$ given by
	\[
		\mu_{\gamma,\mathrm{GMC}}(\mathrm{d}t)=\exp\Big({\gamma\psi(t)-\frac{\gamma^2}{2}\mathbb{E}[\psi^2(t)]}\Big)\mathrm{d}t,
	\]
	where $\{\psi(t)\}_{t\in[0,1]}$ is a centered Gaussian field with a log-correlated covariance kernel
	\begin{align}\label{log-cor}
		\mathbb{E}[\psi(t)\psi(s)]=\log\frac{1}{|t-s|}, \quad t, s \in [0,1]. 
	\end{align}
	The precise definition of $\mu_{\gamma, \mathrm{GMC}}$ is recalled in \S~\ref{sec-gmc}.

	Throughout the whole paper, for each $\gamma\in(0,\sqrt{2})$, we define $D_\gamma \in (0,1)$ by 
	\begin{align}\label{D-gamma}
		D_{\gamma}: =\left\{
		\begin{array}{cl}
			1-\gamma^2 & \text{if  $0<\gamma<\sqrt{2}/2$}
			\vspace{2mm}
			\\
			(\sqrt{2}-\gamma)^2 & \text{if $\sqrt{2}/2\leq\gamma<\sqrt{2}$}
		\end{array}\right..
	\end{align}

	\begin{theorem}\label{FourierDim-GMC}
		For each $\gamma\in(0,\sqrt{2})$, almost surely, we have  $\mathrm{dim}_{F}(\mu_{\gamma,\mathrm{GMC}})=D_{\gamma}$. 
	\end{theorem}

Replacing Bacry-Muzy's white noise decomposition in \S \ref{sec-gmc}  by one white noise decomposition using the structure of the hyperbolic unit disk,  we may obtain the same result as Theorem \ref{FourierDim-GMC} for GMC on the unit circle with the covariance kernel: 
\[
K(\theta, \theta')  = \log \frac{1}{|e^{i\theta} - e^{i\theta' }|}, \quad \theta, \theta' \in [0,2\pi]. 
\]
The methods developped in this paper seem to be also applicable, at least in certain case of the complex GMC model of  Lacoin-Rhodes-Vargas \cite{LRV15}.  This will be the main topic in a seperate paper. 
	
	In the  classical GMC theory, the exact log-kernel \eqref{log-cor} can usually be replaced by the following perturbed log-kernel:    
	\begin{align}\label{bdd-pert}
	\E[\psi(t)\psi(s)]  = \log \frac{1}{|t-s|} + g(s,t),
	\end{align}
	with $g$ being bounded and continuous, and occasionally endowed with greater smoothness (see, for instance,   Junnila-Saksman-Webb \cite{JSW19} and  Huang-Saksman \cite{HS23}).   This perturbed log-kernel retains many important properties of GMC measure and is  important in higher dimension when $d\ge 3$, since for $d =3$, it remains unresolved  whether the exact log-kernel is of $\sigma$-positive type, while for $d \ge 4$, the exact log-kernel is not even of positive type. Therefore, in Part II of this series of works, we adopt  perturbed log-kernels of the form \eqref{bdd-pert} to tackle higher dimensional GMC.   
	
	However,  the mere boundedness and continuity  of $g$ is inadequate for analyzing the exact polynomial Fourier decay of the resultant GMC measure. Indeed, multiplication by a continuous density on a measure  can profoundly change its  Fourier transform's asymptotic behavior.   Therefore,  to ensure the Fourier dimension of the GMC measure remains invariant under the perturbation, further smoothness constraints on $g$ are indispensable.

	\subsection{Frostman regularity and Fourier restriction estimate}
	
	Recall that a non-negative Borel measure $\nu$ on $\R$ is said to be  $\alpha$-upper Frostman regular if 
\[
\sup  \Big\{    \frac{\nu(I)}{|I|^\alpha}:   \text{$I$ are finite  intervals of $\R$}  \Big\} <\infty. 
\]

\begin{corollary}\label{cor-frost}
For each $\gamma\in(0,\sqrt{2})$, almost surely,  $\mu_{\gamma,\mathrm{GMC}}$ is $\alpha$-upper Frostman regular for any  $0\le \alpha< D_\gamma/2$. 
\end{corollary}

The derivation of Corollary~\ref{cor-frost} from Theorem~\ref{FourierDim-GMC} is routine and is omitted in this paper. The reader may refer to  \cite[Corollary~1.5]{CHQW24} for the details of this derivation.  It would be interesting to obtain the optimal exponent of $\alpha$-upper Frostman regularity of $\mu_{\gamma, \GMC}$. 
Note that, the upper-Frostman regularity $\mu_{\gamma,\mathrm{GMC}}$ has also been previously established  by  Astala-Jones-Kupiainen-Saksman    \cite[Theorem~3.7]{AJKS11} in the study of the conformal welding of the random homeomorphism of the unit circle induced by the GMC random measure on the unit circle.

Theorem~\ref{FourierDim-GMC} and Corollary~\ref{cor-frost} combined with the Fourier restriction estimate obtained in \cite[Theorem~4.1]{Moc00}   imply 

\begin{corollary}
For each $\gamma\in(0,\sqrt{2})$, almost surely,  the measure  $\mu_{\gamma,\mathrm{GMC}}$ satisfies the following Fourier restriction estimate: for any $1\le r< \frac{4}{4-D_\gamma}$,  there exists a constant $C(r, \mu_{\gamma,\mathrm{GMC}})>0$ such that for all $f\in L^r(\R)$, 
\[
\| \widehat{f}\|_{L^2(\mu_{\gamma,\mathrm{GMC}})} \le C(r, \mu_{\gamma,\mathrm{GMC}}) \| f\|_{L^r(\R)}. 
\]
\end{corollary}

	\subsection{Outline of the proof of Theorem \ref{FourierDim-GMC}}
	
	The main part of the proof  of Theorem~\ref{FourierDim-GMC} is to establish the almost sure lower bound $\dim_F(\mu_{\gamma, \GMC})\ge D_\gamma$.   The upper bound  $\dim_F(\mu_{\gamma, \GMC})\le D_\gamma$ in this model, as well as in many other models of multiplicative chaos, is relatively simple  and can be obtained in several different manner.  We also mention that, to prove the upper bound of the Fourier dimension, one only needs to study the asymptotic behavior of the Fourier coefficients  along a  given sequence (say, along the dyadic integers $k=2^n$), see \cite{CHQW24} for an application of this  idea. 

\subsubsection{The upper bound}
	
	For the upper bound, we shall use the classical result in potential theory: the Fourier dimension $\dim_F(\nu)$ of a finite Borel measure $\nu$ on any bounded domain of $\R^d$ is dominated by its correlation dimension $\dim_2(\nu)$ (see  \cite[Section~2.6]{BSS23} or \S \ref{sec-fd} below for the various equivalent definitions of $\dim_2(\nu)$): 
	\[
	\dim_F(\nu)\le \dim_2(\nu). 
	\]
	In our situation, for any $\gamma\in(0,\sqrt{2})$,  the following almost sure equality  holds (with $D_\gamma$ given by the formula \eqref{D-gamma}):
	  \begin{align}\label{cor-dim-eq}
	\dim_2(\mu_{\gamma,\mathrm{GMC}}) = D_\gamma.
	\end{align}
Indeed, 	the equality \eqref{cor-dim-eq}  follows from  the multifractal analysis and especially the $L^2$-spectrum or the correlation dimension of GMC, see Bertacco  \cite[Theorem 3.1]{Ber23} for sub-critical GMC in any domain of $\R^d$ with $d\ge 1$.   The equality  \eqref{cor-dim-eq} has already been studied by   Rhodes-Vargas  \cite[Section~4.2]{RV14} and   Garban-Vargas \cite[Remark~2]{GV23}, and  implicitly mentioned  in Lacoin-Rhodes-Vargas \cite{LRV15}.   

Note also that for small parameter values, Garban and Vargas \cite[Theorem 1.3]{GV23} proved a central limit theorem: for  $0<\gamma <\sqrt{2}/2$,  the following convergence holds in law:
\begin{align}\label{CLT-GV}
n^{(1-\gamma^2
)/2}  \widehat{ \mu_{\gamma,\mathrm{GMC}}}(n)  \xrightarrow[n\to\infty]{(d)}  \sqrt{\frac{\kappa}{2}}
\,W_{\mu_{2\gamma,\mathrm{GMC}}([0,1])},
\end{align}
where $W$ is a complex Brownian
motion independent of $\mu_{2\gamma,\mathrm{GMC}}$ and $\kappa>0$ is an explicitly given constant. Note that Garban and Vargas proved the central limit theorem \eqref{CLT-GV} for the GMC on the circle, but their method works for the GMC on the unit interval.     It is then easy to derive from  \eqref{CLT-GV} that for small parameter $0<\gamma < \sqrt{2}/2$,  almost surely,  one has $\dim_F(\mu_{\gamma,\mathrm{GMC}})\le D_\gamma= 1- \gamma^2$.   Therefore, it seems to be of independent interest to study similar convergence in law  as \eqref{CLT-GV} for $\gamma\in [\sqrt{2}/2, \sqrt{2})$,  see \cite[Proposition 1.12]{CHQW24} for a related result.

	\subsubsection{The lower bound via vector-valued martingale method}
	To establish the almost sure lower bound $\dim_F(\mu_{\gamma, \GMC})\ge D_\gamma$, we shall use an important improvement of the vector-valued martingale method  discovered in \cite{CHQW24} in the studies of the Fourier dimensions of Mandelbrot cascade random measures.   Compared with the vector-valued martingale method in the setting of Mandelbrot cascades, several key improvements are required in the setting of GMC:
	\begin{itemize}
	\item We will focus on the upper estimate in  the {\it random Fourier decoupling estimate} obtained in Proposition \ref{dec-vm-bis} below and do not make any effort in proving its sharpness. Therefore, instead of using all the powerful vector-valued martingale inequalities due to Burkholder-Rosenthal or  Bourgain-Stein as in \cite{CHQW24}, here we only need the Pisier's martingale type $p$ inequalities. 
	\item The tree structure together with the independence of the random weights in Mandelbrot cascades will be replaced by the so-called Bacry-Muzy's white noise decomposition of the log-correlated Gaussian field and  an {\it odd-even decomposition} (see  \eqref{def-even-odd} and \eqref{odd-even-dec} for its precise meaning).  Here we shall also mention that, for higher dimensional GMC, in order to obtain greater smoothness of certain stochastic processes (which will be necessary),  the Bacry-Muzy's white noise decomposition should be replaced by a construction of  $\sigma$-regular and $*$-scale invariant kernel to the log-correlated Gaussian field.
	\item Compared with that in  the setting of Mandelbrot cascades,  a key difficulty arises in the GMC setting for obtaining   the {\it separation-of-variable estimate} of the higher-frequency part of the localized Fourier transform (see \eqref{def-YI-intro} for its definition).  We shall resolve this difficulty by applying a  dyadic-discrete-time approximation of the  Gaussian stochastic processes used in defining  the random weights. In particular, the H\"older continuity and the constant of the H\"older continuity of the Gaussian stochastic processes will be crucial for us. 
	\item  The Abel's summation method will play an important role. See Step 4 in the proof of Proposition~\ref{UB-ZW}.  
	\item The discrete version  of the product rule for derivatives will be used. See the elementary identity \eqref{elem-id} and its application in Step 9 in the proof of Proposition~\ref{UB-ZW}. 
	\end{itemize} 
	
In what follows,  we briefly outline the vector-valued  martingale method for obtaining the optimal Fourier decay  in the setting of GMC.

	We shall use the  construction of GMC via the Bacry-Muzy's  white noise decomposition of  Gaussian field with the log-correlated covariance kernel \eqref{log-cor} (this construction will be recalled in \S~\ref{sec-gmc}, see \cite{BKNSW15} and \cite{BM03} for more details). In particular,  for any $\gamma\in (0, \sqrt{2})$, we can define a sequence of independent stochastic processes depending on the parameter $\gamma$ (see the formula \eqref{def-Xm} below for its precise definition)
	\[
	\{ X_m(t) = X_{\gamma, m}(t): t\in [0,1]\}_{m\ge 0}  \quad \text{with $X_m(t)\ge 0$ and $\E[X_m(t)]\equiv 1$}.
	\]
	Then the GMC measure  $\mu_{\gamma, \GMC}$ is defined as the limit of  the following approximating random measures in the sense of weak convergence of measures (see  the formula \eqref{def-mu-gm} below for the details): 
	\[
	\mu_{\gamma,m}(\mathrm{d}t)=  \Big[\prod_{j=0}^{m}X_{j}(t) \Big] \mathrm{d}t. 
	\]
	Note that $(\mu_{\gamma,m})_{m\ge 0}$ is a measure-valued martingale with respect to the natural  filtration. 
	
	Now, for each fixed $\gamma\in(0,\sqrt{2})$ and any fixed $\tau\in(0,D_{\gamma})$, consider the vector-valued martingale $(M_m)_{m\geq0}$ defined by 
	\begin{align}\label{vector-valued-martingale}
		M_m=M_{\gamma,\tau,m}:=(n^{\tau/2}\widehat{\mu_{\gamma,m}}(n))_{n\geq1}.
	\end{align}
That is,  for each $n\geq1$,
	\[
	M_m(n)=n^{\tau/2}\widehat{\mu_{\gamma,m}}(n)=n^{\tau/2}\int_0^1e^{-2\pi int}\mu_{\gamma,m}(\mathrm{d}t).
	\]
	By Lemma~\ref{lem-fm} below, for any integer  $m\ge 0$, if $1<p<2$ and   $q > \frac{4}{1-\tau}$,  then 
	\[
	\E\Big[ \Big\{\sum_{n=1}^{\infty}    |M_m(n)|^q\Big\}^{p/q}\Big]<\infty. 
	\]
	Therefore, $(M_m)_{m\ge 0}$ is $\ell^q$-valued martingale with $\E[\|M_m\|_{\ell^q}^p]<\infty$ for all integers $m\ge 0$.

	The key step in proving Theorem~\ref{FourierDim-GMC}  is the following Theorem~\ref{Uniform-Bound} on the uniform $L^p(\ell^q)$-boundedness of the $\ell^q$-valued martingale $(M_m)_{m\ge 0}$.

	\begin{theorem}\label{Uniform-Bound}
		Let  $\gamma\in(0,\sqrt{2})$ and $\tau\in(0,D_{\gamma})$. Then there exist  $p\in(1,2)$ and $q\in(\frac{4}{1-\tau}, \infty)$ such that 
\[
	(p-1)\Big(1-\frac{\gamma^2p}{2}\Big)-\frac{\tau p}{2}-\frac{p}{q}>0.
	\]
		Moreover, for any such exponents $p$ and $q$,  	we have
		\begin{align}\label{sup-Mm}
			\sup_{m\geq 0}\mathbb{E}[\|M_m\|_{\ell^{q}}^{p}]<\infty.
		\end{align}
	\end{theorem}
	
	By the standard argument in the theory of vector-valued martingales, by \eqref{weacontoGMC}, the inequality \eqref{sup-Mm} is equivalent to  
	\[
	\mathbb{E}\Big[\Big\{\sum_{n=1}^{\infty}|n^{\tau/2} \widehat{\mu_{\gamma,\mathrm{GMC}}}(n) |^{q}\Big\}^{p/q}\Big]<\infty
	\]
	and hence, almost surely, 
	\[
	| \widehat{\mu_{\gamma,\mathrm{GMC}}}(n)|^2 = O(n^{-\tau}) \quad \text{as $n\to\infty$}. 
	\]
Since $\tau\in (0, D_\gamma)$ is chosen arbitrarily,  the above asymptotic relation  provides the desired almost sure lower bound of $\dim_F(\mu_{\gamma,\mathrm{GMC}})\ge D_\gamma$. 
	
	Let us explain our  strategy of the proof of Theorem~\ref{Uniform-Bound}. The three key ingredients are 
	\begin{itemize}
	\item {\it Localization via twice applications of martingale type $p$ inequalities for  $\ell^q$}: In Proposition~\ref{dec-vm-bis} below, by applying twice martingale type $p$ inequalities for the Banach space $\ell^q$, we shall show that
	\[
	\sup_{m\ge 1}\mathbb{E}[\|M_m\|_{\ell^q}^p]\lesssim  \E[\| M_1\|_{\ell^q}^p]  +   \sum_{k=2}^{\infty}\sum_{I \in \mathscr{D}_{k-1}}\mathbb{E}[\|Y_I \|_{\ell^q}^p],
	\]
	where $\mathscr{D}_{k-1}$ is the family of dyadic sub-intervals of $[0,1)$ with length $2^{-(k-1)}$ and for each dyadic interval $I\in \mathscr{D}_{k-1}$, the localized random vector   $Y_I = (Y_I(n))_{n\ge 1}$ is defined by 
	\begin{align}\label{def-YI-intro}
	Y_I(n)=  n^{\frac{\tau}{2}}\int_{I}\Big[\prod_{j=0}^{k-1}X_j(t)\Big]    (X_{k} (t) - 1) e^{-2\pi int}\mathrm{d}t \quad \text{for all $n\ge 1$.}
	\end{align}
	\item {\it Dyadic-discrete-time approximation of the stochastic process in $Y_I(n)$:} Inspired by the classical Littlewood-Paley theory,  for any $I\in \mathscr{D}_{k-1}$ and $2^{L+k-1}< n\le  2^{L+k}$,  we  will use a  dyadic-discrete-time approximation of $(L+k-1)$-level for the stochastic process 
	  \[
D_k(t)=	\Big[\prod_{j=0}^{k-1}X_j(t)\Big]\cdot(X_{k} (t) - 1).
	\]
	More precisely, we divide $I$ into $2^L$ sub-intervals  $J$ of the same length $2^{-(L+k-1)}$:
	\[
	 I= \bigsqcup_{J\in\mathscr{D}_{L+k-1}(I) }J \quad \text{\, with\,} \quad \mathscr{D}_{L+k-1}(I): = \Big\{J\subset I: J\in \mathscr{D}_{L+k-1}\Big\}.
	\] By denoting $\ell_J$ the left end-points of the sub-interval $J$, we decompose $D_k(t)$ as the dyadic approximation term and an error term:  
	\begin{align}\label{dec-dya-err}
	D_k(t) = \underbrace{ \sum_{J\in \mathscr{D}_{L+k-1}(I)}   D_k(\ell_J) \mathds{1}_J(t)}_{\text{dyadic approximation term}}  + \underbrace{ \sum_{J\in \mathscr{D}_{L+k-1}(I)} [D_k(t)-D_k(\ell_J)] \mathds{1}_J(t)}_{\text{error term}}.
	\end{align}
		\item  {\it The  separation-of-variable pointwise upper estimate for $ |Y_I(n)|$}: We shall see that the main decay of $Y_I(n)$ comes from the Fourier transform of the above dyadic approximation term and the oscillatory integrals  
		\[
		\int_J e^{-2 \pi i n t} \mathrm{d}t \quad \text{with $2^{L+k-1}<n\le 2^{L+k}$ and $J\in \mathscr{D}_{L+k-1}(I)$}.
		\] 
Moreover, in summing up the Fourier transforms of each part corresponding to an $J\in \mathscr{D}_{L+k-1}(I)$ in the above dyadic approximation term, we will obtain   a crucial cancellation by using the Abel's summation method (see \eqref{Abel-cancellation} for the precise identity) and will arrive at an upper estimate of the Fourier transforms of the dyadic approximation term as certain  weighted sum of the following differences:
\begin{align}\label{D-diff}
D_k(\ell_J) - D_k(r_J)
\end{align}
with $r_J$ being the right end-point of the sub-interval $J$. 
		
Then, for controlling  the differences $D_k(\ell_J) - D_k(r_J)$ in \eqref{D-diff} and the differences $[D_k(t)-D_k(\ell_J)]\mathds{1}_J(t)$ in the error terms in \eqref{dec-dya-err}, 	we are going to use an appropriate H\"older regularity of the stochastic process	$D_k(t)$ induced by the regularity of $X_j(t)$ (see Lemma \ref{Holder-p-moment-X_m} below).  We shall obtain,  in Proposition~\ref{UB-ZW} below,  the following separation-of-variable pointwise upper estimate 
\begin{align}\label{sep-es}
|Y_I(n)| \le    v_0(n) R_0 +  \sum_{L=1}^\infty v_L(n)  R_L + \sum_{L=1}^\infty w_L(n)  Q_L \quad \text{for all $n\ge 1$,}
\end{align}
where $R_L$, $Q_L$ are non-negative random variables and  $v_L$, $w_L$ are deterministic (without randomness)   sequences of non-negative numbers with supports 
\begin{align}\label{supp-ass}
	\supp(v_0)  =  [1, 2^k] \cap\N \an \supp(v_L)  = \supp(w_L) = (2^{k+L-1}, 2^{k+L}] \cap \N. 
\end{align}
Moreover,  $R_L$, $Q_L$  and  $v_L$, $w_L$ are all explicitly constructed with rather simple forms.  The separation-of-variable pointwise estimate \eqref{sep-es} combined with the conditions \eqref{supp-ass} turns out to be particularly useful for our purpose. Indeed, it allows us to  obtain immediately the upper estimate of $\E[\|Y_I\|_{\ell^q}^p]$:
\[
\E[\|Y_I\|_{\ell^q}^p]\lesssim   \sum_{L=0}^\infty \| v_L\|_{\ell^q}^p \cdot  \E[ R_L^p] + \sum_{L=1}^\infty \| w_L\|_{\ell^q}^p \cdot   \E[ Q_L^p]. 
\]
Then, by estimating all the quantities $ \| v_L\|_{\ell^q}^p$,  $\| w_L\|_{\ell^q}^p$ and $\E[ R_L^p]$,  $\E[ Q_L^p]$, we obtain 
	\[
	\E[\|Y_I\|_{\ell^q}^p] \lesssim    |I|^{p-\frac{p(p-1)\gamma^2}{2}-\frac{\tau p}{2}-\frac{p}{q}}. 
	\]
	\end{itemize}

	\subsection*{Acknowledgements.} 
YQ is supported by National Natural Science Foundation of China (NSFC No. 12471145).

	\section{Preliminaries}

	\subsection{Fourier dimension and correlation dimension}\label{sec-fd}
	In this subsection, we always assume that $\nu$ is a finite positive Borel measure on the unit interval $[0,1]$ and we denote its Fourier transform by
	\[
	\widehat{\nu}(\xi) := \int_{[0,1]} e^{-2\pi i \xi t}  \nu(\mathrm{d} t), \quad \xi \in \mathbb{R}.
	\]
	
	The Fourier dimension of $\nu$ is defined by (see, e.g., \cite[Section~8.2]{BSS23})
	\begin{align}\label{def-Fourier-dim}
	\dim_F(\nu) := \sup \{ D \in [0, 1): |\widehat{\nu}(\xi)|^2 = O(|\xi|^{-D}) \text{ as } \xi \to \infty\}.
	\end{align}
	
	Following Kahane \cite[Chapter~17, Lemma~1]{Kah85b}, we may reduce the study of the decay behavior of the Fourier transform of $\widehat{\nu}(\xi)$ as $\xi \to \infty$ to that of its Fourier coefficients $\widehat{\nu}(n)$ on the positive integers as $n \to \infty$. More precisely,  let $\nu$ be a finite positive Borel measure on $[0, 1]$. Then
	\begin{align}\label{R-to-Z}
		\dim_F(\nu) = \sup \left\{ D \in [0, 1) : |\widehat{\nu}(n)|^2 = O(n^{-D}) \text{ as } n \to \infty \right\}.
	\end{align}
	\begin{remark*}
		We note that, usually in the literature,  the equality \eqref{R-to-Z}   is used for  measures supported on a small sub-interval within $[\delta, 1 - \delta]$ for some $0 < \delta < 1/2$. However, by Kahane's original work, one can remove this assumption on the support.  See \cite[Lemma~1.8]{CHQW24} for the details. 
	\end{remark*}

The correlation dimension $\dim_2(\nu)$, or sometimes called the lower $L^2$-dimension of the measure $\nu$ is defined by (see \cite[Lemma~2.6.6 and Definition~2.6.7]{BSS23})
\begin{align}\label{def-sup-ball}
\dim_2(\nu): =  \liminf_{\delta\to 0^+}\frac{\displaystyle \log \Big(\sup \sum_i \nu \big(B(x_i, \delta)\big)^2 \Big)}{\log \delta},
\end{align}
where the supremum is taken over all families of disjoint balls. The above definition is equivalent to  (see \cite[Definition~2.6.1]{BSS23})
\[
\dim_2(\nu) =  \liminf_{\delta\to 0^+}\frac{\displaystyle \log \Big(\int  \nu \big(B(x, \delta)\big) \nu(\mathrm{d}x) \Big)}{\log \delta}.
\]

The above definition is also equivalent to the following one in terms of the Riesz-energy (see \cite[Proposition~2.1]{HK97}): 
\begin{align}\label{R-en-def}
\dim_2(\nu)= \sup\Big\{s\ge 0: \int\frac{\nu(\mathrm{d}x)\nu(\mathrm{d}y)}{|x-y|^s} <\infty \Big\}.
\end{align}
Since the correlation dimension of a measure  is always dominated by its Hausdorff dimension (see \cite[Theorem~1.4]{FLR02}), in the above definition \eqref{R-en-def}, we may always assume that $0\le s <1$. Then by the standard equality for the Riesz energy (see, e.g., \cite[Theorem~3.10]{Mat15}),  for any $0\le s<1$, 
	\[
	\int_{[0, 1]^2} \frac{\nu(\mathrm{d}x)\nu (\mathrm{d}y)}{|x - y|^{s}} =  \pi^{s-1/2} \frac{\Gamma\big(\frac{1-s}{2}\big)}{\Gamma\big(\frac{s}{2}\big)} \int_\mathbb{R} |\widehat{\nu}(\xi)|^2|\xi|^{s-1}\mathrm{d}\xi.
	\] 
	Hence the definition \eqref{R-en-def} is further equivalent to 
	\begin{align}\label{Fourier-cor-def}
	\dim_2(\nu)= \sup\Big\{0\le s<1 :  \int_\mathbb{R} |\widehat{\nu}(\xi)|^2|\xi|^{s-1}\mathrm{d}\xi <\infty \Big\}.
	\end{align}
	In particular, the above definition \eqref{Fourier-cor-def} for the correlation dimension $\dim_2(\nu)$ compared  with the definition  \eqref{def-Fourier-dim} for the Fourier dimension $\dim_F(\nu)$ implies the following classical inequality: 
	\begin{align}\label{F-less-cor} 
	\dim_F(\nu)\le \dim_2(\nu). 
	\end{align}
	We note that the inequality \eqref{F-less-cor} is also used in \cite{CLS24} in the study of Fourier dimensions of Mandelbrot cascades. 
	
	\subsection{Martingale type $p$ inequalities for $\ell^q$} 
	
	We shall use the following well-known fact in the local theory of Banach spaces, also known as the theory of Banach space geometry (see \cite[Proposition~10.36 and Definition~10.41]{Pis16}): 
	\begin{center}
		{\it  For any  $2\le q<\infty$, the Banach space $\ell^q$ has martingale type $p$ for all $1<p\le 2$. }
	\end{center}
	More precisely,  for any  $1<p\leq2\leq q<\infty$, there exists a constant $C(p, q)>0$ such that any $\ell^q$-valued martingale  
	$(F_m)_{m\ge 0}$ in $L^p(\PP; \ell^q)$ satisfies
	\begin{align}\label{def-Mtype}
		\E [ \| F_m\|_{\ell^q}^p] \le  C(p, q) \sum_{k =0}^m   \E[ \|F_k- F_{k-1}\|_{\ell^q}^p],
	\end{align}
	with the convention $F_{-1}\equiv 0$.

	The inequality \eqref{def-Mtype} implies in particular that for any family of independent and  centered  random variables $(G_k)_{k=0}^m$  in $L^p(\PP; \ell^q)$, 
	\begin{align}\label{def-ind-Mtype}
		\E\Big[ \Big\|\sum_{k=0}^m G_k \Big\|_{\ell^q}^p\Big] \le C(p, q)   \sum_{k =0}^m \E[\|G_k \|_{\ell^q}^p].  
	\end{align} 
	
	\section{Construction of GMC via Bacry-Muzy's white noise decomposition}\label{sec-gmc}
	
	We start with recalling the  construction of GMC via the Bacry-Muzy's  white noise decomposition of  Gaussian field with the log-correlated covariance kernel \eqref{log-cor}. See \cite{BKNSW15} and \cite{BM03} for details.

	Let $\lambda$ be the hyperbolic measure on the upper halp-plane, that is, for any Borel set $A\subset\mathbb{R}\times\mathbb{R}_+$,
	\begin{align}\label{hyperbolic measure}
		\lambda(A):=\int_{A}\frac{\mathrm{d}x\mathrm{d}y}{y^2}.
	\end{align}
	Let $W$ denote the white noise on $\mathbb{R}\times\mathbb{R}_+$ with control measure $\lambda$. In fact, $W$ is considered as a random real function on the Borel sets of $\mathbb{R}\times\mathbb{R}_+$ with finite $\lambda$-measure characterized by the following properties: for all disjoint Borel sets $A,B\subset\mathbb{R}\times\mathbb{R}_+$ satisfying $\lambda(A),\lambda(B)<\infty$, we have
	
	(1) $W(A)$ is a centered Gaussian random variable with variance $\lambda(A)$;
	
	(2) $W(A)$ and $W(B)$ are independent;
	
	(3) $W(A\sqcup B)=W(A)+W(B)$ almost surely.
	
	For any $m\geq0$ and any $t\in[0,1]$, denote the Borel set (see the left part in  Figure \ref{fig-CA})
	\[
	\mathcal{C}_m(t):=\Big\{(x,y)\in\mathbb{R}\times\mathbb{R}_+\,\Big|\,y>\max\big\{2|x-t|,2^{-m}\big\},\,|x-t|<\frac{1}{2}\Big\}
	\]
	and define
	\begin{align}\label{def-psim}
		\psi_{m}(t): =W(\mathcal{C}_m(t)).
	\end{align}
	
\begin{figure}[H]
\centering
\begin{subfigure}{0.25\textwidth}
\centering
\begin{tikzpicture}
  \def\xt{0.2}
  \tikzset{>={Stealth}}
  \begin{axis}[
    axis equal image,
    scale only axis,
    width=5.5cm,
    axis x line=middle,
    axis y line=none,
    xtick={\xt-0.5,0,\xt,\xt+0.5,1},
    xticklabels={$t-\frac{1}{2}$,0\vphantom{$\frac{1}{2}$},$t$\vphantom{$\frac{1}{2}$},$t+\frac{1}{2}$,$1$\vphantom{$\frac{1}{2}$}},
    xmin=-0.5, xmax=1.2,
    ymin=0,    ymax=1.3,
    tick align=outside,
    tick style={thin,black},
    every inner x axis line/.append style={-,thick},
  ]
    \addplot [name path=pp,draw=none] coordinates {
      (\xt-0.5,\pgfkeysvalueof{/pgfplots/ymax}) (\xt-0.5,1) (\xt,0) (\xt+0.5,1) (\xt+0.5,\pgfkeysvalueof{/pgfplots/ymax}) };

    \path [name path=p0] (\pgfkeysvalueof{/pgfplots/xmin},\pgfkeysvalueof{/pgfplots/ymax}) -- (\pgfkeysvalueof{/pgfplots/xmax},\pgfkeysvalueof{/pgfplots/ymax});

    \path [name path=p1] (\pgfkeysvalueof{/pgfplots/xmin},1) -- (\pgfkeysvalueof{/pgfplots/xmax},1);

    \path [name path=p4] (\pgfkeysvalueof{/pgfplots/xmin},0.125) -- (\pgfkeysvalueof{/pgfplots/xmax},0.125);

    \draw [name path=d0,thick, densely dotted, intersection segments={of=pp and p4,sequence=L1 -- R2 -- L3}];

    \draw [thick,densely dotted,intersection segments={of=pp and p4,sequence=L2}];

    \draw [thick,densely dotted] (axis cs:\xt-0.5,0) -- (axis cs:\xt-0.5,1) (axis cs:\xt+0.5,0) -- (axis cs:\xt+0.5,1);

    \addplot [blue!30]  fill between [of=d0 and p0];

    \node  [font=\footnotesize] at (axis cs:\xt,0.9) {$\mathcal{C}_{m}(t)$};

    \draw [thick,densely dotted,intersection segments={of=pp and p1,sequence=R3}];

    \draw [thick,densely dotted,intersection segments={of=pp and p4,sequence=R3}];

    \coordinate (A1) at (axis cs:\xt+0.5,0);
    \coordinate (A2) at (axis cs:\xt+0.5,1);
    \coordinate (D1) at (axis cs:\xt+0.7,0);
    \coordinate (D2) at (axis cs:\xt+0.7,0.125);
  \end{axis}
  \draw [<->] ($(A1)+(1.2em,0)$) -- node [right] {$1$} ($(A2)+(1.2em,0)$);
  \draw [<->] ($(D1)+(1.2em,0)$) -- node [right,font=\footnotesize] {$2^{-m}$} ($(D2)+(1.2em,0)$);
\end{tikzpicture}
\end{subfigure}
\qquad \qquad \quad
\begin{subfigure}{0.45\textwidth}
\centering
\begin{tikzpicture}
  \def\xt{0.2}
  \tikzset{>={Stealth}}
  \begin{axis}[
    axis equal image,
    scale only axis,
    width=6.5cm,
    axis x line=middle,
    axis y line=none,
    xtick={\xt-0.5,0,\xt,\xt+0.5,1},
    xticklabels={$t-\frac{1}{2}$,0\vphantom{$\frac{1}{2}$},$t$\vphantom{$\frac{1}{2}$},$t+\frac{1}{2}$,$1$\vphantom{$\frac{1}{2}$}},
    xmin=-0.5, xmax=1.5,
    ymin=0,    ymax=1.3,
    tick align=outside,
    tick style={thin,black},
    every inner x axis line/.append style={-,thick},
  ]
    \addplot [name path=pp,draw=none] coordinates {
      (\xt-0.5,\pgfkeysvalueof{/pgfplots/ymax}) (\xt-0.5,1) (\xt,0) (\xt+0.5,1) (\xt+0.5,\pgfkeysvalueof{/pgfplots/ymax}) };

    \path [name path=p0] (\pgfkeysvalueof{/pgfplots/xmin},\pgfkeysvalueof{/pgfplots/ymax}) -- (\pgfkeysvalueof{/pgfplots/xmax},\pgfkeysvalueof{/pgfplots/ymax});

    \path [name path=p1] (\pgfkeysvalueof{/pgfplots/xmin},1) -- (\pgfkeysvalueof{/pgfplots/xmax},1);

    \path [name path=p2] (\pgfkeysvalueof{/pgfplots/xmin},0.5) -- (\pgfkeysvalueof{/pgfplots/xmax},0.5);

    \path [name path=p3] (\pgfkeysvalueof{/pgfplots/xmin},0.25) -- (\pgfkeysvalueof{/pgfplots/xmax},0.25);

    \path [name path=p4] (\pgfkeysvalueof{/pgfplots/xmin},0.125) -- (\pgfkeysvalueof{/pgfplots/xmax},0.125);

    \draw [name path=d0,thick, densely dotted, intersection segments={of=pp and p4,sequence=L1 -- R2 -- L3}];

    \path [name path=d1, densely dotted, intersection segments={of=pp and p1,sequence=L1 -- R2 -- L3}];

    \draw [thick, densely dotted, intersection segments={of=pp and p1,sequence=R2}];

    \path [name path=d2, densely dotted, intersection segments={of=pp and p2,sequence=L1 -- R2 -- L3}];

    \draw [thick, densely dotted, intersection segments={of=pp and p2,sequence=R2}];

    \path [name path=d3,densely dotted, intersection segments={of=pp and p3,sequence=L1 -- R2 -- L3}];

    \draw [thick, densely dotted, intersection segments={of=pp and p3,sequence=R2}];

    \draw [thick,densely dotted,intersection segments={of=pp and p4,sequence=L2}];

    \draw [thick,densely dotted] (axis cs:\xt-0.5,0) -- (axis cs:\xt-0.5,1) (axis cs:\xt+0.5,0) -- (axis cs:\xt+0.5,1);

    \coordinate (M1) at (axis cs:\xt-0.2,0.19);
     \coordinate (M2) at (axis cs:\xt,0.19);
     \draw[->,red!60,thick] (M2) -- (M1);

    \addplot  [blue!30]  fill between [of=d1 and p0];

    \addplot  [gray!30]   fill between [of=d2 and d1];

    \addplot  [red!30]   fill between [of=d0 and d3];

    \node  [font=\footnotesize] at (axis cs:\xt,1.15) {$A_{0}(t)$};
    \node  [font=\footnotesize] at (axis cs:\xt,0.75) {$A_{1}(t)$};
    \node  [font=\footnotesize] at (axis cs:\xt,0.40) {$\vdots$};
    \node [red!60, font=\footnotesize]  at (axis cs:\xt-0.34,0.19) {$A_{m}(t)$};

    \draw [thick,densely dotted,intersection segments={of=pp and p1,sequence=R3}];

    \draw [thick,densely dotted,intersection segments={of=pp and p2,sequence=R3}];

    \draw [thick,densely dotted,intersection segments={of=pp and p3,sequence=R3}];

    \draw [thick,densely dotted,intersection segments={of=pp and p4,sequence=R3}];

    \coordinate (A1) at (axis cs:\xt+0.5,0);
    \coordinate (A2) at (axis cs:\xt+0.5,1);
    \coordinate (B1) at (axis cs:\xt+0.7,0);
    \coordinate (B2) at (axis cs:\xt+0.7,0.5);
    \coordinate (C1) at (axis cs:\xt+0.9,0);
    \coordinate (C2) at (axis cs:\xt+0.9,0.25);
    \coordinate (D1) at (axis cs:\xt+1.1,0);
    \coordinate (D2) at (axis cs:\xt+1.1,0.125);
  \end{axis}
  \draw [<->] ($(A1)+(1.2em,0)$) -- node [right,yshift=1em] {$1$} ($(A2)+(1.2em,0)$);
  \draw [<->] ($(B1)+(1.2em,0)$) -- node [right,font=\footnotesize,yshift=1.2em] {$\vdots$} ($(B2)+(1.2em,0)$);
  \draw [<->] ($(C1)+(1.2em,0)$) -- node [right,font=\footnotesize,yshift=0.6em] {$2^{-(m-1)}$} ($(C2)+(1.2em,0)$);
  \draw [<->] ($(D1)+(1.2em,0)$) -- node [right,font=\footnotesize] {$2^{-m}$} ($(D2)+(1.2em,0)$);
\end{tikzpicture}
\end{subfigure}\caption{The  sets $\mathcal{C}_m(t)$ (left)  and  $A_m(t)$ (right).}\label{fig-CA}
\end{figure}

	For any fixed $m\geq0$ (see \cite{BKNSW15}), we have 
	\begin{align*}
	\mathbb{E}[\psi_m(t)\psi_m(s)]&= \lambda(\mathcal{C}_m(t) \cap \mathcal{C}_m(s) ) 
	\\
	& = \left
	\{\begin{array}{ll}
		m\log2+1-2^m|t-s|& \text{if $|t-s|<2^{-m}$} 
		\vspace{2mm}
		\\
		\log\frac{1}{|t-s|} & \text{if $2^{-m}\leq|t-s|\leq1$}
	\end{array}\right..
	\end{align*}
	Note that from the above covariance kernel, we know that the stochastic process $\{\psi_m(t): t\in [0,1]\}$ is translation-invariant.

	For any fixed $\gamma\in(0,\sqrt{2})$, define the random measure $\mu_{\gamma,m}$ on the unit interval $[0,1]$ by
	\begin{align}\label{def-mum}
		\mu_{\gamma,m}(\mathrm{d}t): =\exp\Big({\gamma\psi_m(t)-\frac{\gamma^2}{2}\mathbb{E}[\psi_m^2(t)]}\Big)\mathrm{d}t.
	\end{align}
	This construction fits into the framework of Kahane's theory of Gaussian multiplicative chaos \cite{Kah85a, Kah87}, which implies that almost surely, as $m\to\infty$, the measure $\mu_{\gamma,m}$ tends to the Gaussian multiplicative chaos $\mu_{\gamma,\mathrm{GMC}}$ in the sense of weak convergence of measures:
	\begin{align}\label{weacontoGMC}
	\lim_{m\to \infty}\mu_{\gamma,m}=\mu_{\gamma,\mathrm{GMC}}.
	\end{align}

	Now for any $t\in[0,1]$, consider the Borel set
	\begin{align}\label{def-A0}
		A_0(t): =\Big\{(x,y)\in\mathbb{R}\times\mathbb{R}_+\,\Big|\,y>1,\,|x-t|<\frac{1}{2}\Big\}
	\end{align}
	and for any $m\geq1$,
	\begin{align}\label{def-Am}
	\begin{split}
		A_m(t): & =\mathcal{C}_m(t)\setminus\mathcal{C}_{m-1}(t)
\\
& =\Big\{(x,y)\in\mathbb{R}\times\mathbb{R}_+\,\Big|\,\max\big\{2|x-t|,2^{-m}\big\}<y\leq2^{-(m-1)}\Big\}.
\end{split}
	\end{align}
	See the right part in the Figure \ref{fig-CA} for the illustration of the sets $A_m(t)$. 
	
	For any $m\geq0$,  define the centered Gaussian process  $\varphi_m$ by 
	\begin{align}\label{def-phim}
		\varphi_{m}(t):=W(A_m(t)), \quad t \in [0,1].
	\end{align}
	The family of Borel sets $\{A_m(t):t\in[0,1],\,m\geq0\}$ satisfy  the following elementary properties:
	\begin{itemize}
		\item for any $t\in[0,1]$, all $A_m(t)$, $m\geq0$, are mutually disjoint and
		\[
		\mathcal{C}_m(t)=\bigsqcup_{j=0}^{m}A_j(t);
		\]
		\item  for any $m\neq k$ and any $t,s\in[0,1]$, $A_m(t)$ and $A_k(s)$ are disjoint;
		\item  for any $t_1, t_2\in[0,1]$ satisfying $|t_1-t_2|\geq2^{-(m-1)}$, $A_m(t_1)$ and $A_m(t_2)$ are disjoint. Consequently, for any sub-intervals $T,S\subset[0,1]$ satisfying
			\[
			\mathrm{dist}(T,S) = \inf\{|t_1-t_2|: t_1\in T \an t_2 \in S\} \geq  2^{-(m-1)},
			\]
the two subsets $\bigcup\limits_{t_1 \in T} A_m(t_1)$ and $\bigcup\limits_{t_2\in S} A_m(t_2)$ are disjoint.  See Figures \ref{t-12-dis} and \ref{TS-dis} for the illustrations. 
	\end{itemize}

	\begin{figure}[H]
	\centering
	\begin{tikzpicture}
  \def\xta{0.27}
  \def\xtb{0.73}
  \tikzset{>={Stealth}}
  \begin{axis}[
    axis equal image,
    scale only axis,
    width=8.5cm,
    axis x line=middle,
    axis y line=none,
    xtick={0,\xta,\xtb,1},
    xticklabels={0,$t_{1}$,$t_{2}$,$1$},
    xmin=-0.05, xmax=1.05,
    ymin=0,     ymax=0.275,
    tick align=outside,
    tick style={thin,black},
    every inner x axis line/.append style={-,thick},
  ]
    \draw [name path=p1,thick,densely dotted] (\pgfkeysvalueof{/pgfplots/xmin},0.25)  -- (\pgfkeysvalueof{/pgfplots/xmax},0.25);

    \draw [name path=p2,thick,densely dotted] (\pgfkeysvalueof{/pgfplots/xmin},0.125) -- (\pgfkeysvalueof{/pgfplots/xmax},0.125);

    \addplot [name path=g1,samples at={0,\xta,1},draw=none]{2*abs(x-\xta)};

    \addplot [name path=g2,samples at={0,\xtb,1},draw=none]{2*abs(x-\xtb)};

    \draw [name path=d1,thick, densely dotted, intersection segments={of=g1 and p2,sequence=L1 -- R2 -- L3}];

    \draw [name path=d2,thick,densely dotted, intersection segments={of=g2 and p2,sequence=L1 -- R2 -- L3}];

    \draw [name path=t1,thick, densely dotted, intersection segments={of=g1 and p1,sequence=R2}];

    \draw [name path=t2,thick, densely dotted, intersection segments={of=g2 and p1,sequence=R2}];

    \draw [thick,densely dotted,intersection segments={of=p2 and g1,sequence=R2}];

    \draw [thick,densely dotted,intersection segments={of=p2 and g2,sequence=R2}];

    \addplot  [red!30]  fill between [of=d1 and t1];

    \addplot  [blue!30] fill between [of=d2 and t2];

    \node at (axis cs:\xta,0.19) {$A_{m}(t_{1})$};

    \node at (axis cs:\xtb,0.19) {$A_{m}(t_{2})$};

    \addplot[mark=*,only marks,mark size=1.2pt] coordinates {(0.375,0) (0.625,0)};

    \coordinate (A1) at (axis cs:0.95,0);
    \coordinate (A2) at (axis cs:0.95,0.25);
    \coordinate (B1) at (axis cs:1,0);
    \coordinate (B2) at (axis cs:1,0.125);
    \coordinate (C1) at (axis cs:0.375,0);
    \coordinate (C2) at (axis cs:0.625,0);
  \end{axis}
  \draw [<->] ($(A1)+(0.6em,0)$) -- node [right,yshift=1em] {$2^{-(m-1)}$} ($(A2)+(0.6em,0)$);
  \draw [<->] ($(B1)+(0.6em,0)$) -- node [right,font=\footnotesize] {$2^{-m}$} ($(B2)+(0.6em,0)$);
  \draw ($(C1)+(0,-1.6em)$) -- (C1);
  \draw ($(C2)+(0,-1.6em)$) -- (C2);
  \draw [<->] ($(C1)+(0,-0.8em)$) -- node [fill=white,font=\footnotesize] {$2^{-(m-1)}$} ($(C2)+(0,-0.8em)$);
\end{tikzpicture}\caption{$|t_1-t_2|\geq2^{-(m-1)} \Longrightarrow A_m(t_1)$ and $A_m(t_2)$ are disjoint.}\label{t-12-dis}
\end{figure}

\begin{figure}[H]
\begin{tikzpicture}
  \def\xta{0.27}
  \def\xtb{0.73}
  \tikzset{>={Stealth}}
  \begin{axis}[
    axis equal image,
    scale only axis,
    width=8.5cm,
    axis x line=middle,
    axis y line=none,
    xtick={0,1},
    xticklabels={0,1},
    xmin=-0.05, xmax=1.05,
    ymin=0,     ymax=0.275,
    tick align=outside,
    tick style={thin,black},
    every inner x axis line/.append style={-,thick},
  ]

    \path [name path=p0] (\pgfkeysvalueof{/pgfplots/xmin},0)  -- (\pgfkeysvalueof{/pgfplots/xmax},0);

    \draw [name path=p1,thick,densely dotted] (\pgfkeysvalueof{/pgfplots/xmin},0.25)  -- (\pgfkeysvalueof{/pgfplots/xmax},0.25);

    \draw [name path=p2,thick,densely dotted] (\pgfkeysvalueof{/pgfplots/xmin},0.125) -- (\pgfkeysvalueof{/pgfplots/xmax},0.125);

    \addplot [name path=g1,samples at={0,\xta,1},draw=none]{2*abs(x-\xta)-0.21};

    \addplot [name path=g2,samples at={0,\xtb,1},draw=none]{2*abs(x-\xtb)-0.21};

    \draw [name path=d1,thick, densely dotted, intersection segments={of=g1 and p2,sequence=L1 -- R2 -- L3}];

    \draw [name path=d2,thick, densely dotted, intersection segments={of=g2 and p2,sequence=L1 -- R2 -- L3}];

    \draw [name path=t1,thick,densely dotted, intersection segments={of=g1 and p1,sequence=R2}];

    \draw [name path=t2,thick,densely dotted, intersection segments={of=g2 and p1,sequence=R2}];

    \draw [thick,densely dotted,intersection segments={of=p2 and g1,sequence=R2}];

    \draw [thick,densely dotted,intersection segments={of=p2 and g2,sequence=R2}];

    \addplot  [red!30] fill between [of=d1 and t1];

    \addplot   [blue!30] fill between [of=d2 and t2];

    \node at (axis cs:\xta,0.19) {$\bigcup\limits_{t_1\in T}A_{m}(t_1)$};

    \node at (axis cs:\xtb,0.19) {$\bigcup\limits_{t_2\in S}A_{m}(t_2)$};

    \coordinate (A1) at (axis cs:0.95,0);
    \coordinate (A2) at (axis cs:0.95,0.25);
    \coordinate (B1) at (axis cs:1,0);
    \coordinate (B2) at (axis cs:1,0.125);
    \coordinate (C1) at (axis cs:0.375,0);
    \coordinate (C2) at (axis cs:0.625,0);

    \coordinate (D1) at (axis cs:2*\xta-0.375,0);
    \coordinate (D2) at (axis cs:2*\xtb-0.625,0);

  \end{axis}
  \draw [<->] ($(A1)+(0.6em,0)$) -- node [right,yshift=1em] {$2^{-(m-1)}$} ($(A2)+(0.6em,0)$);
  \draw [<->] ($(B1)+(0.6em,0)$) -- node [right,font=\footnotesize] {$2^{-m}$} ($(B2)+(0.6em,0)$);
  \draw ($(C1)+(0,-1.6em)$) -- (C1);
  \draw ($(C2)+(0,-1.6em)$) -- (C2);
  \draw [<->] ($(C1)+(0,-0.8em)$) -- node [fill=white,font=\footnotesize] {$2^{-(m-1)}$} ($(C2)+(0,-0.8em)$);
  \draw[decoration={brace,mirror,raise=2pt},decorate,thick,red!60]  (D1) -- node[below=6pt,font=\footnotesize] {$T$} (C1);
  \draw[decoration={brace,mirror,raise=2pt},decorate,thick,blue!60] (C2) -- node[below=6pt,font=\footnotesize] {$S$} (D2);
  \fill (C1) circle (1.2pt);
  \fill (C2) circle (1.2pt);
\end{tikzpicture}\caption{$\mathrm{dist}(S,T)\geq2^{-(m-1)}\Longrightarrow \bigcup\limits_{t_1\in T}A_m(t_1)\cap  \bigcup\limits_{t_2\in S}A_m(t_2) = \varnothing$.}\label{TS-dis}
\end{figure}

	For any fixed $\gamma\in(0,\sqrt{2})$ and any $m\geq0$, define  the  stochastic process 
	\begin{align}\label{def-Xm}
		X_{m}(t)=X_{\gamma,m}(t): =\exp\Big({\gamma\varphi_m(t)-\frac{\gamma^2}{2}\mathbb{E}[\varphi_m^2(t)]}\Big), \quad t\in [0,1].
	\end{align}
	Then the properties on the family $\{A_m(t):t\in[0,1],\,m\geq0\}$ imply the following elementary properties of $\psi_m$, $\varphi_m$ and $X_m$. 
	
	\begin{properties}\label{elem-prop}
		The stochastic processes $\psi_m$, $\varphi_m$ and $X_m$ satisfy  the following properties:
		\begin{itemize}
			\item[(P1)] The functions $\psi_m$, $\varphi_m$ defined in \eqref{def-psim} and \eqref{def-phim} satisfy 
			\[
				\psi_m(t)=\sum_{j=0}^{m}\varphi_{j}(t)
			\]
			and hence the measure $\mu_{\gamma, m}$ defined in \eqref{def-mum} can be written as 
			\begin{align}\label{def-mu-gm}
				\mu_{\gamma,m}(\mathrm{d}t)=\exp\Big(\sum_{j=0}^{m}\big({\gamma\varphi_j(t)-\frac{\gamma^2}{2}\mathbb{E}[\varphi_j^2(t)]}\big)\Big)\mathrm{d}t =  \Big[\prod_{j=0}^{m}X_{j}(t) \Big] \mathrm{d}t;
			\end{align}
			\item[(P2)] The stochastic processes $\{\varphi_{m}\}_{m\geq0}$ are  independent and hence so are the processes $\{X_m\}_{m\geq0};$  
			\item[(P3)] For any sub-intervals $T,S\subset[0,1]$ satisfying
			\[
			\mathrm{dist}(T,S) = \inf\{|t_1-t_2|: t_1\in T \an t_2 \in S\} \geq  2^{-(m-1)},
			\]
			the stochastic processes $\{\varphi_{m}(t):t\in T\}$ and $\{\varphi_{m}(t): t\in S\}$ are independent and hence so are the pairs $\{X_m(t):t\in T\}$ and $\{X_m(t): t\in S\}$.
		\end{itemize}
	\end{properties}

	\begin{lemma}\label{p-moment-X_m}
		For any $\gamma\in(0,\sqrt{2})$, any $t\in[0,1]$ and any $p>0$, we have
		\[
		\mathbb{E}[X_j^p(t)]=\left\{\begin{array}{cl}
			e^{\frac{p(p-1)\gamma^2}{2}} &  \text{if $j=0$}
			\vspace{2mm}
			\\
			2^{\frac{p(p-1)\gamma^2}{2}} & \text{if $j\geq1$}
		\end{array}\right..
		\]
	\end{lemma}
	\begin{proof}
		Fix any $\gamma\in(0,\sqrt{2})$, any $t\in[0,1]$ and any $p>0$.  By the definitions \eqref{def-A0}, \eqref{def-Am},  \eqref{def-phim}, \eqref{def-Xm} for $A_j(t)$,  $\varphi_j(t)$ and $X_j(t)$ respectively,   for any  $j\ge 0$,  we have 
		\[
		\mathbb{E}[\varphi_j^2(t)]=\lambda(A_j(t)) \an 
		X_j^p(t)=\exp\Big({p\gamma\varphi_j(t)-\frac{p\gamma^2}{2}\mathbb{E}[\varphi_j^2(t)]}\Big).
		\]
		Hence 
		\[
		\mathbb{E}[X_j^p(t)]=\exp\Big({\frac{p^2\gamma^2}{2}\mathbb{E}[\varphi_j^2(t)]-\frac{p\gamma^2}{2}\mathbb{E}[\varphi_j^2(t)]}\Big)=\exp\Big(\frac{p(p-1)\gamma^2}{2} \lambda(A_j(t))\Big).
		\]
		Now  by the definition \eqref{hyperbolic measure} of the hyperbolic measure $\lambda$ on the upper half-plane and the definitions \eqref{def-A0}, \eqref{def-Am} for the subsets $A_j(t)$,   
		\begin{align}\label{two-area}
			\lambda(A_0(t))= 1 \an \lambda(A_j(t)) = \log 2  \quad \text{if  $j\ge 1$}. 
		\end{align}
		This completes the proof of the lemma.
	\end{proof}

	\begin{lemma}\label{Holder-p-moment-X_m}
		For any $\gamma\in(0,\sqrt{2})$  and any $p>0$, we have
		\[
		\sup_{j\geq0}\sup_{0<|t-s|\leq 2^{-j} }\mathbb{E}\Big[\Big| \frac{X_j(t)-X_j(s)}{\sqrt{2^{j}  |t-s|}} \Big|^p\Big]<\infty.
		\]
	\end{lemma}
	
	\begin{lemma}\label{lem-two-gaussian}
		Let $g_1$, $g_2$ be i.i.d. standard normal random variables. Then for any $p>0$,  there exists a constant $C_p>0$ such that  for any $\sigma \in [0,\sqrt{2}]$, 
		\[
		\E [| \exp(\sigma g_1) - \exp(\sigma g_2) |^p] \le  C_p \sigma^p. 
		\]
	\end{lemma}
	\begin{proof}
		By Lagrange's mean value theorem for the function $x\mapsto e^x$, for all $\sigma\in [0,\sqrt{2}]$, we have  
		\[
		| \exp(\sigma g_1) - \exp(\sigma g_2) |\le  \sigma  | g_1-g_2|  \exp(\sqrt{2}|g_1| + \sqrt{2}|g_2|). 
		\]
		We complete the proof by taking $C_p = \E[| g_1-g_2|^p \exp(p\sqrt{2}|g_1| +p\sqrt{2} |g_2|) ] <\infty$. 
	\end{proof} 
	
	\begin{proof}[Proof of Lemma~\ref{Holder-p-moment-X_m}]
		Since the stochastic process $X_j$ is translation-invariant, we only need to deal with $X_j(t) - X_j(0)$.  Fix $j\ge 0$ and $0<t\le 2^{-j}$.   Define three independent centered Gaussian random variables  (see Figure \ref{3-ind} for an illustration) as 
		\[
		\xi_{\mathrm{left}}(t) =  W(A_j(t)\setminus A_j(0)), \quad  \xi_{\mathrm{right}}(t)  =  W(A_j(0)\setminus A_j(t))
		\]
		and 
		\[
		  \xi_{\mathrm{center}}(t) = W(A_j(t) \cap A_j(0)). 
		\]
		\begin{figure}[H]
		\begin{tikzpicture}
  \def\xta{0.035}
  \def\xtb{0.090}
  \tikzset{>={Stealth}}
  \begin{axis}[
    axis equal image,
    scale only axis,
    width=5.5cm,
    axis x line=middle,
    axis y line=none,
    xtick={\xta,\xtb,\xta+0.125},
    xticklabels={$0$\vphantom{$2^{-m}$},$t$\vphantom{$2^{-j}$},$2^{-j}$},
    xmin=-0.125, xmax=0.375,
    ymin=0,     ymax=0.275,
    tick align=outside,
    tick style={thin,black},
    every inner x axis line/.append style={-,thick},
  ]
    \draw [name path=p1,thick,densely dotted] (\pgfkeysvalueof{/pgfplots/xmin},0.25)  -- (\pgfkeysvalueof{/pgfplots/xmax},0.25);

    \draw [name path=p2,thick,densely dotted] (\pgfkeysvalueof{/pgfplots/xmin},0.125) -- (\pgfkeysvalueof{/pgfplots/xmax},0.125);

    \addplot [name path=g1,draw=none] coordinates {(\xta-0.5,1) (\xta,0) (\xta+0.5,1)};

    \addplot [name path=g2,draw=none] coordinates {(\xtb-0.5,1) (\xtb,0) (\xtb+0.5,1)};

    \addplot [name path=g3,draw=none] coordinates {(\xtb-0.5,1) ((\xta+\xtb)/2,\xtb-\xta) (\xta+0.5,1)};

    \addplot [name path=g4,draw=none] coordinates {(\xta-0.5,1) ((\xta+\xtb)/2,\xta-\xtb) (\xtb+0.5,1)};

    \draw [name path=d1,thick,intersection segments={of=g1 and p2,sequence=L1 -- R2 -- L3}];
    \draw [name path=t1,thick,intersection segments={of=g1 and p1,sequence=R2}];
    \draw [name path=d2,thick,intersection segments={of=g2 and p2,sequence=L1 -- R2 -- L3}];
    \draw [name path=t2,thick,intersection segments={of=g2 and p1,sequence=R2}];
    \path [name path=d3,thick,intersection segments={of=g3 and p2,sequence=L1 -- R2 -- L3}];
    \draw [name path=t3,thick,intersection segments={of=g3 and p1,sequence=R2}];

    \addplot [red!30]   fill between [of=d1 and t1];
    \addplot [gray!30] fill between [of=d2 and t2];
    \addplot [blue!30]  fill between [of=d3 and t3];

    \draw [thick,densely dotted,intersection segments={of=p2 and g1,sequence=R2}];

    \draw [thick,densely dotted,intersection segments={of=p2 and g2,sequence=R2}];

    \coordinate (A1) at (axis cs:\xta-0.0625,0.1875);
    \coordinate (A2) at (axis cs:\xta-0.15,0.325);
    \coordinate (B1) at (axis cs:\xtb+0.0625,0.1875);
    \coordinate (B2) at (axis cs:\xtb+0.15,0.325);

    \coordinate (C1) at (axis cs:0.0625,0.1875);
    \coordinate (C2) at (axis cs:0.0625,0.275);

    \coordinate (D1) at (axis cs:0.25,0);
    \coordinate (D2) at (axis cs:0.25,0.25);
    \coordinate (E1) at (axis cs:0.3,0);
    \coordinate (E2) at (axis cs:0.3,0.125);

    \coordinate (F1) at (axis cs:0,0);
    \coordinate (F2) at (axis cs:0.125,0);

  \end{axis}
  \draw[->,red!60,thick] (A1) -| (A2);
  \draw[->,gray,thick] (B1) -| (B2);
  \draw[->,blue!60,thick] (C1) -- (C2);
  \draw [<->] ($(D1)+(0.6em,0)$) -- node [right,yshift=1em] {$2^{-(j-1)}$} ($(D2)+(0.6em,0)$);
  \draw [<->] ($(E1)+(0.6em,0)$) -- node [right,font=\footnotesize] {$2^{-j}$} ($(E2)+(0.6em,0)$);
  \node [anchor=south,red!60]  [font=\footnotesize]  at (A2) {$A_{j}(0)\setminus A_{j}(t)$};
  \node [anchor=south,gray] [font=\footnotesize]  at (B2) {$A_{j}(t)\setminus A_{j}(0)$};
  \node [anchor=south,blue!60] [font=\footnotesize]  at (C2) {$A_{j}(t)\cap A_{j}(0)$};
\end{tikzpicture}\caption{The regions corresponding to $\xi_{\mathrm{left}}(t), \xi_{\mathrm{right}}(t), \xi_{\mathrm{center}}(t)$.}\label{3-ind}
		\end{figure}

		\noindent Then we have
		\begin{align}\label{var-2jt}
			\Var (\xi_{\mathrm{left}}(t)) = \Var(\xi_{\mathrm{right}}(t))    = \lambda(A_j(0)\setminus A_j(t)) =\left\{\begin{array}{ll}
				t &  \text{if $j=0$}
				\vspace{2mm}
				\\
				2^{j-1} t & \text{if $j\geq1$}
			\end{array}\right.
		\end{align}
		and by \eqref{two-area}, 
		\[
		\Var(\xi_{\mathrm{center}}(t)) =  \lambda(A_j(t) \cap A_j(0)) \le \lambda(A_j(0)) \le 1. 
		\]
		Then, by writing  $
		\varphi_j(t) = \xi_{\mathrm{left}}(t) + \xi_{\mathrm{center}}(t)$, $\varphi_j(0)  = \xi_{\mathrm{right}}(t)+\xi_{\mathrm{center}}(t)$ and noting that  
		\[
		\E[\varphi_j^2(t)] =\E[\varphi_j^2(0)] =  \lambda (A_j(0)), 
		\]
		we obtain 
		\begin{align*}
			X_j(t)- X_j(0)&  =  \exp\Big({\gamma\varphi_j(t)-\frac{\gamma^2}{2}\mathbb{E}[\varphi_j^2(t)]}\Big)  - \exp\Big({\gamma\varphi_j(0)-\frac{\gamma^2}{2}\mathbb{E}[\varphi_j^2(0)]}\Big)
			\\
			& = \Big[ \exp (\gamma \xi_{\mathrm{left}}(t) )  -  \exp (\gamma \xi_{\mathrm{right}}(t))\Big] \exp\Big(\gamma  \xi_{\mathrm{center}}(t)-\frac{\gamma^2}{2}\lambda(A_j(0))\Big). 
		\end{align*}
		Therefore, by the independence of $\xi_{\mathrm{left}}(t)$, $\xi_{\mathrm{right}}(t)$, $\xi_{\mathrm{center}}(t)$, we obtain 
		\begin{align*}
		& \E[|X_j(t) - X_j(0)|^p]   
		\\
		=&  \E\Big[\Big| \exp (\gamma \xi_{\mathrm{left}}(t) )  -  \exp (\gamma \xi_{\mathrm{right}}(t))\Big|^p\Big]  \cdot \E\Big[\exp\Big(p\gamma \xi_{\mathrm{center}}(t)-\frac{p\gamma^2}{2}\lambda(A_j(0))\Big)\Big]. 
		\end{align*}
		On the one hand, we have 
		\begin{align*}
			\E\Big[\exp\Big(p\gamma \xi_{\mathrm{center}}(t)-\frac{p\gamma^2}{2}\lambda(A_j(0))\Big)\Big]
			&  =    \exp\Big(\frac{p^2\gamma^2 \Var(\xi_{\mathrm{center}}(t))}{2} - \frac{p\gamma^2 \lambda(A_j(0))}{2}\Big)
			\\
			& \le \exp\Big(\frac{p(p-1)\gamma^2 \lambda(A_j(0))}{2}\Big) 
			\\
			& \le \exp\Big(\frac{p(p-1)\gamma^2}{2}\Big).  
		\end{align*}
		On the other hand, note that  $\xi_{\mathrm{left}}(t)$ and $\xi_{\mathrm{right}}(t)$ are i.i.d. centered Gaussian random variables,  by Lemma~\ref{lem-two-gaussian} and \eqref{var-2jt},  for any $\gamma\in (0, \sqrt{2})$ and any $0<t\le 2^{-j}$,  we have 
		\[
		\sigma: = \gamma \sqrt{\Var(\xi_{\mathrm{left}}(t))} \le \sqrt{2}  \cdot \sqrt{ 2^{{j}} t}=\sqrt{2^{j+1} t} \le \sqrt{2}. 
		\]
		Hence we get
		\begin{align*}
		\E\Big[\big| \exp (\gamma \xi_{\mathrm{left}}(t) )  -  \exp (\gamma \xi_{\mathrm{right}}(t))\big|^p\Big]  & \le C_p \Big(\gamma  \sqrt{\Var(\xi_{\mathrm{left}}(t))}\Big)^p
		\\
		&  \le   C_p ( \sqrt{2^{j+1} t} )^p 
        \\
        &   = \sqrt{2}^{\,p}C_p 2^{jp/2} t^{p/2}. 
		\end{align*}
		Then the desired inequality follows immediately.
	\end{proof}
	
	\begin{corollary}\label{cor-sup-X}
		For any $\gamma\in(0,\sqrt{2})$, any $q>2$ and any integer $j\ge 0$, there exists a modification $\widetilde{X}_j$ of $X_j$ such that 
		\begin{align}\label{Xj-q}
			\E\Big[ \Big( \sup_{t\ne s} \frac{|\widetilde{X}_j(t)-\widetilde{X}_j(s)|}{|t-s|^\alpha}\Big)^{q} \Big] <\infty
		\end{align}
		for any
		\[
		0\le \alpha < \frac{1}{2}- \frac{1}{q}.
		\]
		In particular, we have 
		\[
		\E\Big[\sup_{t\in [0,1]} \widetilde{X}_j(t)^q\Big]<\infty. 
		\]
	\end{corollary}
	\begin{proof}
		Fix any $\gamma\in(0,\sqrt{2})$, any $q>2$ and any integer $j\ge 0$.  Lemma~\ref{Holder-p-moment-X_m} implies that  when $|t-s|<2^{-j}$, 
		\begin{align}\label{av-holder}
			\E[|X_j(t) - X_j(s)|^q] \le C  2^{jq/2} |t-s|^{q/2} = C(j,q)\cdot |t-s|^{q/2}. 
		\end{align}
		By changing the constant $C(j,q)$ if necessary, the inequality \eqref{av-holder} holds for all pairs $(t, s)\in [0,1]^2$. 
		The standard Kolmogorov's continuity theorem (see, e.g., \cite[Chapter~I, Theorem~2.1]{RY99}) now implies the desired inequality. 
	\end{proof}
	
	\begin{corollary}\label{cor-holder}
		For any $\gamma\in(0,\sqrt{2})$, any $\alpha\in [0, 1/2)$ and any integer $j\ge 0$, there exists a modification  $\widetilde{X}_j$ of $X_j$ such that $\widetilde{X}_j$ is H\"older continuous of order $\alpha$ and hence 
		\[
		\prod_{j=0}^m \widetilde{X}_j (t)
		\]
		is H\"older continuous of order $\alpha$ with respect to $t\in [0,1]$.
	\end{corollary}
	
	\begin{proof}
		It follows immediately from Corollary~\ref{cor-sup-X}. 
	\end{proof}

	\begin{convention}\label{conv-mod}
		By Corollary~\ref{cor-holder},  in what follows, given any $\varepsilon>0$, we shall always assume that the stochastic process $X_j(t)$ are H\"older continuous of order $1/2- \varepsilon$ and satisfies the inequality \eqref{Xj-q}.  
	\end{convention}

	\section{Initial steps in the proof of Theorem~\ref{Uniform-Bound}}

	\subsection{An elementary identity}
	The following elementary identity will play a key role  in several places of this paper: given any two finite sequences of complex numbers $(a_j)_{j=0}^m$ and $(b_j)_{j =0}^m$, we have 
	\begin{align}\label{elem-id}
		\prod_{j=0}^{m}a_j-\prod_{j=0}^{m}b_j=\sum_{r=0}^{m}\Big(\prod_{j=0}^{r-1}b_j\Big) \big(a_r-b_r\big) \Big(\prod_{j=r+1}^{m}a_j\Big). 
	\end{align}

	\subsection{A very rough estimate of $\widehat{\mu_{\gamma, m}}$}

	\begin{lemma}\label{lem-fm}
		For each fixed $\gamma\in(0,\sqrt{2})$ and any fixed $\tau\in(0,D_{\gamma})\subset(0,1)$, let $m\ge 0$ be an integer. For any $q > \frac{4}{1-\tau}$, 
		\[
		\E\Big[  \sum_{n=1}^{\infty} | n^{\tau/2} \widehat{\mu_{\gamma, m}} (n)|^q  \Big]<\infty.  
		\]
		In particular, for any $p\in (1,2)$ and $q>\frac{4}{1-\tau}$, we have 
		\[
		\E\Big[ \Big\{  \sum_{n=1}^{\infty} |n^{\tau/2} \widehat{\mu_{\gamma, m}} (n)|^q \Big\}^{p/q} \Big]<\infty.  
		\]
	\end{lemma}
	
	\begin{proof}
		We shall use Convention~\ref{conv-mod}.  
		In view of the defining formula \eqref{def-mu-gm} for the random measure $\mu_{\gamma,m}$, we define   
		\[
		\rho_m(t) = \prod_{j=0}^m X_j(t), \quad t\in [0,1]. 
		\]
		The elementary identity \eqref{elem-id} implies that 
		\[
		|\rho_m(s) - \rho_m(t) |\le   \sum_{r=0}^m  \Big[\prod_{j=0}^{r-1}X_j(t)\Big] \big|X_r(s)-X_r(t)\big| \Big[\prod_{j=r+1}^{m}X_j(s)\Big].
		\]
		Then by defining the independent random variables
		\[
		K_j = \sup_{t\in [0,1]} X_j(t), 
		\]
		we obtain,  for any  $q>2$,  
		\[
		\E[|\rho_m(s) - \rho_m(t)|^q] \le (m+1)^q     \sum_{r=0}^m  \Big(\prod_{j=0}^{r-1} \E[K_j^q]\Big) \E[|X_r(s)-X_r(t)|^q] \Big(\prod_{j=r+1}^{m}\E[K_j^q]\Big).
		\]
		Therefore,  by Lemma~\ref{Holder-p-moment-X_m} and Corollary~\ref{cor-sup-X},  there exists a constant $C(m, q)>0$ such that 
		\[
		\E[|\rho_m(s) - \rho_m(t)|^q] \le   C(m,q)  \cdot  |s-t|^{q/2}.
		\]
		Then, again by the standard Kolmogorov's continuity theorem (see, e.g., \cite[Chapter~I, Theorem~2.1]{RY99}), there exists a modification $\widetilde{\rho}_m$
		of $\rho_m$  such that 
		\begin{align}\label{mod-rho}
			\E \Big[  \Big(\sup_{s\ne t} \frac{|\widetilde{\rho}_m(s) - \widetilde{\rho}_m(t)|}{|s-t|^\alpha}\Big)^q\Big]<\infty
		\end{align}
		for any $0\le \alpha < \frac{1}{2} - \frac{1}{q}$.  Indeed, since both $\rho_m$ and $\widetilde{\rho}_m$ are continuous, they are indistinguishable and hence the inequality \eqref{mod-rho} holds for $\rho_m$ itself.

		Recall that, by definition, the modulus of continuity of $\rho_m$ is given by 
		\[
		\omega(\rho_m, \delta) = \sup_{t,t+\delta\in [0,1]} | \rho_m(t+\delta) - \rho_m(t) |.  
		\]
		For $q>2$, define a random variable 
		\[
		\Lambda_\alpha   = \sup_{ 0<\delta<1}   \frac{\omega(\rho_m, \delta)}{\delta^\alpha}\quad\text{with $0\le \alpha < \frac{1}{2} - \frac{1}{q}$} .
		\]
		The inequality \eqref{mod-rho} for the function $\rho_m$ implies  
		\[
		\E [\Lambda_\alpha^q]<\infty  \quad \text{for all $0\le \alpha < \frac{1}{2} - \frac{1}{q}$}.
		\]
		
		Finally,  for any  $n\ge 1$, 
		\[
		\widehat{\mu_{\gamma, m}}(n)= \widehat{\rho_m}(n). 
		\]
		Therefore, by the standard fact in harmonic analysis  (see, e.g., \cite[Chapter~I, Section~4.6]{Kat04}), we know that 
		\[
		|\widehat{\rho_m}(n)| \le \frac{1}{2}  \omega\Big(\rho_m, \frac{\pi}{n}\Big) \le  \frac{\pi^\alpha}{2 n^\alpha} \Lambda_\alpha. 
		\]
		Then, for any  $q>\frac{4}{1-\tau}$,  there exists $\varepsilon>0$ small enough such that  by taking 
		\[
		0<\alpha = \frac{1}{2}- \frac{1}{q} - \varepsilon< \frac{1}{2}- \frac{1}{q}, 
		\]
		we have 
		\[
		(\alpha- \frac{\tau}{2})q = 1+ \frac{1-\tau}{2} \Big(q - \frac{4}{1-\tau}\Big) - q\varepsilon>1
		\]
		and hence 
		\[
		\E\Big[  \sum_{n=1}^{\infty}| n^{\tau/2} \widehat{\mu_{\gamma, m}} (n)|^q  \Big] \leq C \E[\Lambda_\alpha^q]   \sum_{n=1}^{\infty}  \frac{1}{n^{(\alpha- \tau/2)q}} <\infty.
		\]
		This completes the proof of the lemma. 
	\end{proof}

	\subsection{Existence of suitable exponents for $L^p(\ell^q)$}
	
	\begin{lemma}\label{existence-p0q0}
		For each fixed $\gamma\in(0,\sqrt{2})$ and any fixed $\tau\in(0,D_{\gamma})\subset(0,1)$, there exist $p=p(\gamma,\tau)$ and $q=q(\gamma,\tau)$ satisfying $1<p<2<\frac{4}{1-\tau}<q<\infty$ such that
		\[
		(p-1)\Big(1-\frac{\gamma^2p}{2}\Big)-\frac{\tau p}{2}-\frac{p}{q}>0.
		\]
	\end{lemma}
	
	\begin{proof}
		Consider the function
		\[
		f_\gamma(p) =-\Big(\gamma^2p+\frac{2}{p}\Big)+2+\gamma^2,\quad p\in(1,2].
		\]
		If $0< \gamma<\sqrt{2}/2$, then $f_\gamma$ is increasing in $(1,2]$; while for $\sqrt{2}/2 \le \gamma <\sqrt{2}$, the function $f_\gamma$ is increasing in $(1, \sqrt{2}/\gamma]$ and decreasing in $(\sqrt{2}/\gamma, 2]$. Consequently, 
		\[
		\sup_{p\in(1,2]}f_\gamma(p)
		=\left\{\begin{array}{cl}
			f_\gamma(2) = 1-\gamma^2 & \text{if $0<\gamma<\sqrt{2}/2$}
			\vspace{2mm}
			\\
			f_\gamma(\sqrt{2}/\gamma) = (\sqrt{2}-\gamma)^2 & \text{if  $\sqrt{2}/2\leq\gamma<\sqrt{2}$}
		\end{array}\right..
		\]
		In other words, by \eqref{D-gamma},
		\[
		\sup_{p\in(1,2]}f_\gamma(p) = D_\gamma. 
		\]
	Hence,  for any  $\tau\in(0,D_{\gamma})$, there exists $p_0=p_0(\gamma,\tau)\in(1,2)$ such that $f_\gamma(p_0)>\tau$. It follows that 
		\[
		\frac{p_0}{2}[f_{\gamma}(p_0)-\tau]=(p_0-1)\Big(1-\frac{\gamma^2p_0}{2}\Big)-\frac{\tau p_0}{2}>0.
		\]
		Then for large enough $q_0=q_0(\gamma,\tau)>\frac{4}{1-\tau}>2$,  we have 
		\[
		(p_0-1)\Big(1-\frac{\gamma^2p_0}{2}\Big)-\frac{\tau p_0}{2}-\frac{p_0}{q_0}>0.
		\]
		This completes the proof of the lemma.
	\end{proof}

\section{The main part of the Proof of Theorem~\ref{Uniform-Bound}}

This section is devoted to the proof of  Theorem~\ref{Uniform-Bound}.

{\flushleft \bf Convention of notation:} 
For simplifying notation, in what follows, we will always fix $\gamma\in(0,\sqrt{2})$ and $\tau\in(0,D_{\gamma})$, and fix $p\in(1,2)$, $q\in(\frac{4}{1-\tau},\infty)$ satisfying
\begin{align}\label{p-q-tau}
	\Theta(\gamma, \tau, p, q) := (p-1)\Big(1-\frac{\gamma^2p}{2}\Big)-\frac{\tau p}{2}-\frac{p}{q}>0.
\end{align}
The existence of such a pair $(p,q)$ is guaranteed by Lemma~\ref{existence-p0q0}.   

Then, for instance, in defining the random vector $Y_I$ in the formula \eqref{def-Y} below, instead of writing $Y_I^{(\gamma, \tau)}$, we only write $Y_I$.   Similarly, by writing $A\lesssim  B$, we mean that there exists a finite constant $C>0$ which only depends on the parameters  $\gamma$, $\tau$, $p$, $q$ such that $A \le C B$.

\subsection{The dyadic decomposition and martingale type $p$ inequalities for $\ell^q$}
For each integer $m \ge 0$, let $\mathscr{D}_m$ denote the family of  dyadic sub-intervals of $[0,1)$ of level/generation $m$: 
\begin{align}\label{def-dya}
	\mathscr{D}_m:= \Big\{I\subset [0,1) : I = \Big[\frac{h-1}{2^m}, \frac{h}{2^m}\Big) \, \text{for some  integer $h= 1, \cdots, 2^m$} \Big\}. 
\end{align}
Recall the definition \eqref{def-Xm} of the stochastic processes $\{X_j(t): t\in [0,1]\}_{j\ge 0}$ and  define 
\begin{align}\label{def-Fm}
	\mathscr{F}_m:=\sigma\big(X_j:0\leq j\leq m\big),\quad m\geq 0.
\end{align}
Recall also the definition \eqref{vector-valued-martingale} of the $\ell^q$-valued martingale $(M_m)_{m\ge 0}$ with respect to the natural increasing filtration $(\mathscr{F}_m)_{m\ge 0}$ (see Lemma~\ref{lem-fm} for its $L^p(\ell^q)$-integrability): 
\[
M_m=(n^{\frac{\tau}{2}}\widehat{\mu_{\gamma,m}}(n))_{n\geq1}.
\]

Now for each $k\ge 2$, and for any  dyadic interval $I\in \mathscr{D}_{k-1}$,  we define an $\mathscr{F}_k$-measurable  random vector $Y_I =Y_I^{(\gamma, \tau)}:= (Y_I(n))_{n\ge 1}$ by
\begin{align}\label{def-Y}
	Y_I(n):= n^{\frac{\tau}{2}}\int_{I}\Big[\prod_{j=0}^{k-1}X_j(t)\Big]  \mathring{X}_{k} (t) e^{-2\pi int}\mathrm{d}t,
	\end{align}
	where 
	\[
	\mathring{X}_{k}(t):= X_{k}(t) - \E[X_{k}(t)] = X_{k}(t)-1. 
\]

{\flushleft\bf Alarming:}  One should note that for each $I\in \mathscr{D}_{k-1}$, the random vector $Y_I$  defined in \eqref{def-Y} is $\mathscr{F}_k$-measurable (but is not $\mathscr{F}_{k-1}$-measurable). It is worthwhile to note that, by the item (P2) of Elementary Properties~\ref{elem-prop}, we have 
\begin{align}\label{Y-I-cond}
	\E[Y_I|\mathscr{F}_{k-1}] =0. 
\end{align}

\begin{proposition}\label{dec-vm-bis}
	There exists a constant $C=C(\gamma,\tau,p,q)>0$ such that for any $m\geq 2$,
	\[
	\mathbb{E}[\|M_m\|_{\ell^q}^p]\leq C \E[\| M_1\|_{\ell^q}^p]  +  C \sum_{k=2}^{m}\sum_{I \in \mathscr{D}_{k-1}}\mathbb{E}[\|Y_I \|_{\ell^q}^p].
	\]
\end{proposition}

\begin{remark*}
	We already know from Lemma~\ref{lem-fm} that $\E[\|M_1\|_{\ell^q}^p]<\infty$.  
\end{remark*}

We postpone the detailed proof of Proposition~\ref{dec-vm-bis}  in \S~\ref{sec-2-Mtype}, but here we explain its main ingredients.

The  proof of   Proposition~\ref{dec-vm-bis}   relies crucially on  twice applications of  the martingale type $p$ inequalities for the Banach space $\ell^q$ with $1<p\leq2\leq q<\infty$ and is outlined as follows: 
\begin{itemize}
	\item Firstly, we use the martingale type $p$ inequality \eqref{def-Mtype} and obtain 
	\[
	\mathbb{E}[\|M_m\|_{\ell^q}^p]\lesssim  \E[\|M_1\|_{\ell^q}^p] + \sum_{k=2}^{m}\mathbb{E}[\|M_{k}-M_{k-1}\|_{\ell^q}^p].
	\]
	\item Then, fix an integer $k$ with $2\le k\le m$.  We shall use the dyadic decomposition  of the martingale difference $M_k-M_{k-1}$ into the summation of the random vectors $Y_I$ introduced in \eqref{def-Y}: 
	\[
		M_k-M_{k-1} =  \sum_{I\in \mathscr{D}_{k-1}} Y_I. 
	\]
	For using again the martingale type $p$ inequality, we need to consider the {\it odd-even decomposition} of 
\[
	\mathscr{D}_{k-1}= \mathscr{D}_{k-1}^{\odd} \sqcup \mathscr{D}_{k-1}^\even
\]
	 with  $\mathscr{D}_{k-1}^{\mathrm{odd}}$ and $\mathscr{D}_{k-1}^{\mathrm{even}}$ being the sub-families of $\mathscr{D}_{k-1}$  defined in \eqref{def-even-odd} below.  Then we have 
\[
	M_k-M_{k-1} =  \sum_{I\in \mathscr{D}_{k-1}^\odd} Y_I +   \sum_{I\in \mathscr{D}_{k-1}^\even} Y_I. 
\]
	Now a crucial observation is that, by the item (P3) of the Elementary Properties~\ref{elem-prop}, conditioned on $\mathscr{F}_{k-1}$,   the random vectors $(Y_I)_{I\in \mathscr{D}_{k-1}^\odd}$ are independent (and also conditionally centered by \eqref{Y-I-cond}) and hence with respect to the conditional expectation 
	$\E_{k-1}[\cdot] = \E[\cdot |\mathscr{F}_{k-1}]$, we may apply the martingale type $p$ inequality \eqref{def-ind-Mtype} to obtain 
	\[
	\E_{k-1}\Big[\Big\| \sum_{I\in \mathscr{D}_{k-1}^\odd} Y_I\Big\|_{\ell^q}^p \Big] \lesssim     \sum_{I\in \mathscr{D}_{k-1}^\odd} \E_{k-1}[\| Y_I\|_{\ell^q}^p ]. 
	\]
A similar inequality holds for the summand contributed by $I\in \mathscr{D}_{k-1}^\even$. 
\end{itemize}

\subsection{The localized estimate via separation-of-variable estimate}\label{sec-sep}
The next goal is to estimate $\E[\|Y_I\|_{\ell^q}^p]$.   

Recall the definition \eqref{p-q-tau} for the quantity $\Theta(\gamma, \tau, p,q)$:
\[
\Theta(\gamma, \tau, p,q)  = (p-1)\Big(1-\frac{\gamma^2p}{2}\Big)-\frac{\tau p}{2}-\frac{p}{q}>0. 
\]
 We have the following estimate.
\begin{proposition}\label{UB-ZW}
	There exists a constant $C = C(\gamma, \tau, p, q)>0$ such that for any dyadic sub-interval $I\subset [0,1)$ of generation $k-1$ with $k\geq2$,  
	\[
	\E[\|Y_I\|_{\ell^q}^p] \le C |I|^{ 1 + \Theta(\gamma, \tau, p,q)}. 
	\]
In other words,   for each $k\ge 2$ and any $I\in \mathscr{D}_{k-1}$, 
\[
\E[\|Y_I\|_{\ell^q}^p] \lesssim  2^{-k \cdot \theta(\gamma, \tau, p, q)}
\]
with 
	\[
	\theta(\gamma, \tau, p, q) = 1+ \Theta(\gamma, \tau, p, q)  =  p-\frac{p(p-1)\gamma^2}{2}-\frac{\tau p}{2}-\frac{p}{q}. 
	\]
\end{proposition}

\begin{remark*}
	Clearly, it is of crucial importance that, in Proposition~\ref{UB-ZW}, the constant $C$ is uniform for all dyadic sub-intervals $I\subset [0,1)$. 
\end{remark*}

The proof of Proposition~\ref{UB-ZW}  is much more involved and is postponed to \S~\ref{sec-local}.  

The main steps in the proof of Proposition~\ref{UB-ZW}   are outlined as follows:   take  any $k\ge 2$ and any dyadic interval $I\in \mathscr{D}_{k-1}$, recall that 
\[
\|Y_I\|_{\ell^q}^p = \Big\{\sum_{n=1}^\infty   |Y_I(n)|^q\Big\}^{p/q}. 
\]
The key in obtaining the desired upper estimate of $\E[\|Y_I\|_{\ell^q}^p]$ is to establish  the following {\it separation-of-variable} pointwise upper estimate for $ |Y_I(n)|$ (which is inspired by the standard Littlewood-Paley decomposition in harmonic analysis): 
\[
|Y_I(n)| \le    v_0(n) R_0 +  \sum_{L=1}^\infty v_L(n)  R_L + \sum_{L=1}^\infty w_L(n)  Q_L \quad \text{for all integers $n\ge 1$,}
\]
where $R_L$, $Q_L$ are non-negative random variables and  $v_L$, $w_L$ are deterministic (without randomness)  sequences of non-negative numbers with supports 
\[
\supp(v_0)  = [1, 2^k]\cap \N \an \supp(v_L)  = \supp(w_L)   =  (2^{k+L-1}, 2^{k+L}] \cap\N  \quad  \text{for all $L\ge 1$}. 
\]

In particular, for any $L\ge 1$, the constructions of $v_L$, $w_L$ and $R_L$, $Q_L$  rely on a dyadic-discrete-time approximation of the stochastic process  
\[
\Big[\prod_{j=0}^{k-1} X_j(t)\Big]  \mathring{X}_k(t), \quad t\in I. 
\]
The  level of the discrete-time approximation being  dependent on each dyadic interval for the integer numbers $2^{L+k-1}<n\le 2^{L+k}$.  It should also be mentioned that, such approximation is reasonable (meaning that the difference can be controlled) by Lemma~\ref{Holder-p-moment-X_m}.

\subsection{Derivation of Theorem ~\ref{Uniform-Bound} from  Proposition~\ref{dec-vm-bis}  and  Proposition~\ref{UB-ZW}}

We know from Lemma~\ref{lem-fm} that $\E[\|M_m\|_{\ell^q}^p]<\infty$ for all fixed $m\ge0$. Therefore, by Proposition~\ref{dec-vm-bis}, to prove  Theorem~\ref{Uniform-Bound}, it suffices to prove the inequality 
\begin{align}\label{summable-YI}
	\sum_{k=2}^{\infty}\sum_{I \in \mathscr{D}_{k-1}}\mathbb{E}[\|Y_I \|_{\ell^q}^p]<\infty.
\end{align}
This inequality follows from Proposition~\ref{UB-ZW}. Indeed, since $\# \mathscr{D}_{k-1} = 2^{k-1}$ for any $k\ge 2$,  by Proposition~\ref{UB-ZW}, we have 
\[
\sum_{k=2}^{\infty}\sum_{I \in \mathscr{D}_{k-1}}\mathbb{E}[\|Y_I \|_{\ell^q}^p] \lesssim \sum_{k=2}^{\infty}  2^k\cdot   2^{-k  [1+\Theta(\gamma, \tau, p, q)]}  = \sum_{k=2}^\infty 2^{-k\cdot \Theta(\gamma, \tau, p, q)}. 
\]
Thus,  by \eqref{p-q-tau}, we have $\Theta(\gamma, \tau, p, q)>0$ and   the above geometric series is convergent. Hence  we obtain the desired inequality \eqref{summable-YI}.

\subsection{Proof of Proposition~\ref{dec-vm-bis}}\label{sec-2-Mtype}
Consider the   $\ell^q$-valued martingale $(M_m)_{m\geq 0}$ defined in \eqref{vector-valued-martingale} with respect to the increasing filtration \eqref{def-Fm}.  For any $m\geq 2$, write $M_m$ as the sum of martingale differences: 
\[
M_m =M_1 + \sum_{k=2}^{m}(M_{k}-M_{k-1}).
\]
By the martingale type $p$ inequality \eqref{def-Mtype} for the Banach space $\ell^q$, we have
\begin{align}\label{first-martingale-type-inequality}
	\mathbb{E}[\|M_m\|_{\ell^q}^p]\lesssim  \E[\|M_1\|_{\ell^q}^p] + \sum_{k=2}^{m}\mathbb{E}[\|M_{k}-M_{k-1}\|_{\ell^q}^p].
\end{align}

For each $k\ge 2$ and $n\geq1$,  by the definition  \eqref{vector-valued-martingale}, we have  
\[
M_k(n)=n^{\frac{\tau}{2}}\widehat{\mu_{\gamma,k}}(n)=n^{\frac{\tau}{2}}\int_0^1e^{-2\pi int}\mu_{\gamma,k}(\mathrm{d}t)=n^{\frac{\tau}{2}}\int_0^1\Big[\prod_{j=0}^{k}X_j(t)\Big]e^{-2\pi int}\mathrm{d}t.
\]
Hence we obtain
\[
M_{k}(n)-M_{k-1}(n)=n^{\frac{\tau}{2}}\int_0^1\Big[\prod_{j=0}^{k-1}X_j(t)\Big]\mathring{X}_k(t)e^{-2\pi int}\mathrm{d}t,
\]
where
\[
	\mathring{X}_k(t)=X_k(t)-1.
\]
Recalling the notation $\mathscr{D}_{k-1}$ introduced in \eqref{def-dya} for the family of dyadic intervals and the definition \eqref{def-Y} for the random vector $Y_I$,  we obtain 
\[
M_{k}(n)-M_{k-1}(n)=   \sum_{I\in \mathscr{D}_{k-1}} Y_I(n) \quad \text{for all $n\ge 1$}. 
\]
That is, as random vectors, we have the equality 
\[
M_k - M_{k-1}=   \sum_{I\in \mathscr{D}_{k-1}} Y_I. 
\]

Our next step is to introduce the odd-even decomposition of $\mathscr{D}_{k-1}$: 
\[
\mathscr{D}_{k-1} = \mathscr{D}_{k-1}^\odd \sqcup \mathscr{D}_{k-1}^\even
\] 
with $\mathscr{D}_{k-1}^\odd$ and $\mathscr{D}_{k-1}^\even$ are two sub-families of $\mathscr{D}_{k-1}$ defined as 
\begin{align}\label{def-even-odd}
	\begin{split}
		\mathscr{D}_{k-1}^\odd &= \Big\{I\subset [0,1) : I = \Big[\frac{h-1}{2^{k-1}}, \frac{h}{2^{k-1}}\Big) \, \text{for some odd integer $1\le h \le 2^{k-1}$} \Big\},
		\\
		\mathscr{D}_{k-1}^\even &= \Big\{I\subset [0,1) : I = \Big[\frac{h-1}{2^{k-1}}, \frac{h}{2^{k-1}}\Big) \, \text{for some even integer $1\le h \le 2^{k-1}$} \Big\}.
	\end{split}
\end{align} 
It follows that 
\begin{align}\label{odd-even-dec}
	M_k - M_{k-1}=   \sum_{I\in \mathscr{D}_{k-1}^\odd} Y_I +  \sum_{I\in \mathscr{D}_{k-1}^\even}Y_I
\end{align}
and hence 
\begin{align}\label{D-D-dec}
	\E[\|M_{k}-M_{k-1}\|_{\ell^q}^p] \lesssim  \E\Big[ \Big\|  \sum_{I\in \mathscr{D}_{k-1}^\odd} Y_I \Big\|_{\ell^q}^p\Big] +  \E\Big[ \Big\|  \sum_{I\in \mathscr{D}_{k-1}^\even} Y_I \Big\|_{\ell^q}^p\Big]. 
\end{align}

{\flushleft \bf Key observation:} In the odd-even decomposition  \eqref{odd-even-dec}, for any two distinct intervals $I \ne  I'$ in  the family $\mathscr{D}_{k-1}^\odd$, the distance $\dist(I, I')$ satisfies 
\[
\dist(I, I') \ge 2^{-(k-1)}. 
\]
Therefore, by the item (P3) of Elementary Properties~\ref{elem-prop}, conditioned on $\mathscr{F}_{k-1}$, the random vectors $(Y_I)_{I\in \mathscr{D}_{k-1}^\odd}$ are independent. Moreover, by \eqref{Y-I-cond}, conditioned on $\mathscr{F}_{k-1}$, the random vectors $Y_I$ are centered.    The same holds for the random vectors indexed by $I\in \mathscr{D}_{k-1}^\even$. 

\medskip

Using the above Key observation,  with respect to the conditional expectation 
\[
\E_{k-1}[\cdot] = \E[\cdot |\mathscr{F}_{k-1}],
\]  we may apply the martingale type $p$ inequality  \eqref{def-ind-Mtype} for the Banach space $\ell^q$ to obtain 
\[
\E_{k-1}\Big[\Big\| \sum_{I\in \mathscr{D}_{k-1}^\odd} Y_I\Big\|_{\ell^q}^p \Big] \lesssim     \sum_{I\in \mathscr{D}_{k-1}^\odd} \E_{k-1}[\| Y_I\|_{\ell^q}^p ]
\]	
and hence, by taking expectation on both side, we  get 
\begin{align}\label{odd-type}
	\E  \Big[\Big\| \sum_{I\in \mathscr{D}_{k-1}^\odd} Y_I\Big\|_{\ell^q}^p \Big] \lesssim     \sum_{I\in \mathscr{D}_{k-1}^\odd} \E [\| Y_I\|_{\ell^q}^p ]. 
\end{align}
By exactly the same argument, we have
\begin{align}\label{even-type}
	\E  \Big[\Big\| \sum_{I\in \mathscr{D}_{k-1}^\even} Y_I\Big\|_{\ell^q}^p \Big] \lesssim     \sum_{I\in \mathscr{D}_{k-1}^\even} \E [\| Y_I\|_{\ell^q}^p ]. 
\end{align}

Finally,  by  combining the inequalities \eqref{first-martingale-type-inequality}, \eqref{D-D-dec}, \eqref{odd-type} and \eqref{even-type},   we get  the desired inequality 
\[
\mathbb{E}[\|M_m\|_{\ell^q}^p]\lesssim  \E[\| M_1\|_{\ell^q}^p]  +    \sum_{k=2}^{m}\sum_{I \in \mathscr{D}_{k-1}}\mathbb{E}[\|Y_I \|_{\ell^q}^p].
\]

\subsection{Proof of Proposition~\ref{UB-ZW}}\label{sec-local}

The proof of Proposition~\ref{UB-ZW} is divided into the following twelve steps. 

Fix any $k\ge 2$ and any dyadic interval $I\in \mathscr{D}_{k-1}$.  As explained before in \S~\ref{sec-sep}, our first goal is  to establish  a {\it separation-of-variable} pointwise upper estimate for $ |Y_I(n)|$ introduced in \eqref{def-Y}.

\medskip
{\flushleft \it Step 1. The lower-frequency part $1\le n \le 2^k$.}
\medskip

For the lower-frequency part $1\le n \le 2^k$,  the quantities $Y_I(n)$ are controlled by the total mass  of $\mu_{\gamma,k}$ on the interval $I$. More precisely,  here we use a  very rough upper estimate of $Y_I(n)$: 
\[
|Y_I(n)|=  \Big|  n^{\frac{\tau}{2}}\int_{I}\Big[\prod_{j=0}^{k-1}X_j(t)\Big]  \mathring{X}_{k} (t) e^{-2\pi int}\mathrm{d}t  \Big| \le      n^{\frac{\tau}{2}}\int_{I}\Big[\prod_{j=0}^{k-1}X_j(t)\Big]  |  \mathring{X}_{k} (t)| \mathrm{d}t. 
\]
Hence by defining
\begin{align}\label{def-v-0}
	v_0(n): = n^{\frac{\tau}{2}}\cdot \mathds{1}(1\le n \le 2^k)
\end{align}
and 
\begin{align}\label{def-R-0}
	R_0: = \int_{I}\Big[\prod_{j=0}^{k-1}X_j(t)\Big]  |  \mathring{X}_{k} (t)| \mathrm{d}t,
\end{align}
we obtain 
\begin{align}\label{low-Y}
	|Y_I(n)| \le v_0(n)  R_0 \quad \text{for all $1\le n \le 2^k$}.
\end{align}

\medskip
{\flushleft \it Step 2. Dyadic-discrete-time approximation for the higher-frequency part.}
\medskip

For the higher-frequency part $Y_I(n)$ with $n>2^k$,  we shall use a  finer estimate by applying a dyadic-discrete-time approximation of the stochastic process   
\begin{align}\label{def-Dk}
	D_k(t): = \Big[\prod_{j=0}^{k-1} X_j(t)\Big]  \mathring{X}_k(t), \quad t\in I. 
\end{align}
Namely, we shall approximate $D_k(t)$ by the value of $D_k$ at some dyadic $t$. It is important for our purpose to use a  finer approximation of $D_k(t)$. That is, to control $Y_I(n)$, the level of the dyadic-discrete-time approximation depends on each dyadic interval of integers $(2^{L+k-1},  2^{L+k}]\cap \N$ containing  $n$.

More precisely,  given any integer $L\ge 1$,  by using the same dyadic decomposition of $I$, we shall decompose $Y_I(n)$  in the same  manner for all integers $2^{L+k-1}<n\le 2^{L+k}$.   That is, we divide the dyadic interval  $I\in \mathscr{D}_{k-1}$  into $2^L$ equal pieces (hence each sub-interval has length $2^{-(L+k-1)}$). In other words, denote by  $\mathscr{D}_{L+k-1}(I)$ the family of sub-intervals $J \subset I$ in $\mathscr{D}_{L+k-1}$:
\begin{align}\label{DLk-1I}
	\mathscr{D}_{L+k -1}(I): = \Big\{J \subset I:   J \in \mathscr{D}_{L+k-1} \Big\}. 
\end{align}
By using  the decomposition 
\[
I= \bigsqcup_{J\in \mathscr{D}_{L+k-1}(I)} J, 
\]
we can  decompose $Y_I(n)$ as 
\[
Y_I(n)  =  n^{\frac{\tau}{2}}\int_{I}D_{k} (t) e^{-2\pi int}\mathrm{d}t  = \sum_{J\in \mathscr{D}_{L+k-1}(I)}  n^{\frac{\tau}{2}}\int_{J} D_{k} (t) e^{-2\pi int}\mathrm{d}t. 
\]

Then on each interval $J\in \mathscr{D}_{L+k-1}(I)$, we approximate $D_k(t)$ with $D_k$ evaluated on the left end-point of $J$.   That is, by writing $\ell_J$ the left end-point of the interval $J$ and using the decomposition 
\[
D_k(t)  =  [D_k(t)- D_k(\ell_J)] + D_k(\ell_J) ,  
\]
we obtain 
\begin{align}\label{Y-UV}
\begin{split}
	Y_I(n)  =&   \underbrace{ \sum_{J\in \mathscr{D}_{L+k-1}(I)}  n^{\frac{\tau}{2}}\int_{J} [D_{k} (t) -D_k(\ell_J)]e^{-2\pi int}\mathrm{d}t}_{\text{denoted $V_I(n)$}}
	\\
	&  \quad + \underbrace{\sum_{J\in \mathscr{D}_{L+k-1}(I)}  n^{\frac{\tau}{2}}\int_{J}  D_k(\ell_J) e^{-2\pi int}\mathrm{d}t}_{\text{denoted $U_I(n)$}}. 
	\end{split}
\end{align}
The two terms $V_I(n)$ and $U_I(n)$ will be controlled by different methods.

\medskip
{\flushleft \it Step 3. The simple control of $V_I(n)$.}
\medskip

The term $V_I(n)$ is controlled directly by using the triangle inequality: 
\[
|V_I(n)| \le n^{\frac{\tau}{2}} \sum_{J\in \mathscr{D}_{L+k-1}(I)}  \int_{J} |D_{k} (t) -D_k(\ell_J) | \mathrm{d}t. 
\]
Hence by defining 
\begin{align}\label{def-v-l}
	v_L(n) : = n^{\frac{\tau}{2}}\cdot  \mathds{1}(2^{L+k-1} <n \le 2^{L+k})
\end{align}
and
\begin{align}\label{def-Rl}
	R_L: =   \sum_{J\in \mathscr{D}_{L+k-1}(I)}  \int_{J} |D_{k} (t) -D_k(\ell_J) | \mathrm{d}t, 
\end{align}
we obtain 
\begin{align}\label{V-wQ}
	|V_I(n)| \le v_L(n) R_L \quad \text{for all $2^{L+k-1} <n \le 2^{L+k}$}. 
\end{align}

\begin{remark*}
	It should be emphasized that  the random variable  $R_L$ defined as above  depends on $L$ (and of course it depends on $k$, which is determined by $I$), but does not depend on $n$.   In other words, all integers $2^{L+k-1} <n \le 2^{L+k}$ share the same $R_L$. 
\end{remark*}

\medskip
{\flushleft \it Step 4. The Abel's summation method for controlling $U_I(n)$.}
\medskip

We shall apply the Abel's summation method to the term $U_I(n)$. Indeed, ordering the dyadic sub-intervals $J\in \mathscr{D}_{L+k-1}(I)$ from left to right according to their natural ordering on the real line,  we get 
\[
J_l = [t_{l-1}, t_l), \quad 1\le l \le 2^{L},
\]
with $t_0 = \ell_I$, $t_{2^L} = r_I$,  the left and right end-points of $I$ respectively and 
\[
t_l  - t_{l-1} = |J_l|= \frac{|I|}{2^L} = 2^{-(L+k-1)}, \quad \text{i.e.}\quad 
t_l = \ell_I +  l\cdot 2^{-(L+k-1)} \quad \text{for all $0 \le l \le 2^L$}. 
\]
Under the above notation, by using the elementary equality 
\[
\int_a^b e^{-2\pi i n t}\mathrm{d}t =  \frac{e^{-2\pi i n b} - e^{-2\pi i n a}}{- 2\pi in }, 
\]
we obtain 
\[
U_I(n) = \sum_{l=1}^{2^L}  n^{\frac{\tau}{2}} \int_{J_l} D_k(t_{l-1}) e^{-2\pi i nt}\mathrm{d}t =  \frac{n^{\frac{\tau}{2}}}{- 2 \pi in} \sum_{l=1}^{2^L} D_k(t_{l-1})  [e^{-2\pi i n t_l} - e^{-2 \pi i n t_{l-1}}]. 
\]
An application of Abel's summation method then yields
\begin{align}\label{Abel-cancellation}
U_I(n) =  \frac{n^{\frac{\tau}{2}}}{- 2 \pi in} \Big(  D_k(t_{2^L-1}) e^{-2\pi i n t_{2^L}}  - D_k(t_0) e^{-2\pi i n t_0} +   \sum_{l=1}^{2^L-1}   \big[D_k(t_{l-1}) - D_k(t_l)\big] e^{- 2\pi i n t_l} \Big).  
\end{align}
It follows that 
\[
|U_I(n)|\le \frac{n^{\frac{\tau}{2}-1}}{2\pi}  \Big(  |D_k(t_{2^L-1})|  + |D_k(t_0)| + \sum_{l=1}^{2^L-1} | D_k(t_{l-1}) - D_k(t_l)|\Big). 
\]
Hence by defining 
\begin{align}\label{def-wl}
	w_L(n) : = n^{\frac{\tau}{2}-1} \cdot \mathds{1}(2^{L+k-1} <n \le 2^{L+k})
\end{align}
and
\begin{align}\label{def-Ql}
	Q_L: = \frac{1}{2\pi}  \Big(|D_k(t_{2^L-1})|  + |D_k(t_0)| + \sum_{l=1}^{2^L-1} | D_k(t_{l-1}) - D_k(t_l)|\Big), 
\end{align}
we obtain 
\begin{align}\label{U-vR}
	|U_I(n)| \le w_L(n) Q_L \quad \text{for all $2^{L+k-1} <n \le 2^{L+k}$}. 
\end{align}

\medskip
{\flushleft \it Step 5. Separation-of-variable estimate of $Y_I(n)$.}
\medskip

Combining  \eqref{low-Y}, \eqref{Y-UV},  \eqref{V-wQ} and \eqref{U-vR}, we obtain  the desired separation-of-variable estimate
\begin{align}\label{sep-Y}
	|Y_I(n)| \le    \sum_{L=0}^\infty v_L(n) R_L + \sum_{L=1}^\infty w_L(n) Q_L \quad \text{for all integers $n\ge 1$,}
\end{align}
where $R_L$, $Q_L$ are non-negative random variables and  $v_L$, $w_L$ are deterministic   sequences of non-negative numbers with supports 
\[
	\supp(v_0)  =  [1, 2^k] \cap\N \an \supp(v_L)  = \supp(w_L) = (2^{k+L-1}, 2^{k+L}] \cap \N  \quad  \text{for all $L\ge 1$}. 
\]

\medskip
{\flushleft \it Step 6.  Upper estimate of $\E[\|Y_I\|_{\ell^q}^p]$ via separation-of-variable.}
\medskip

We are going to use   the following elementary inequality (since $0<p/q<1$): 
\begin{align}\label{xy-small}
	\Big(\sum_{i=1}^\infty x_i \Big)^{p/q} \le  \sum_{i=1}^\infty x_i^{p/q} \quad \text{for any  non-negative numbers $x_i\ge 0$}.
\end{align}
Since $v_L$'s in the separation-of-variable estimate  \eqref{sep-Y} have disjoint supports (and so do $w_L$'s),  by using the elementary inequality $(x+y)^q\le 2^q x^q + 2^q y^q$ for all non-negative real numbers $x,y$, we obtain that for any $n\geq1$,
\begin{align*}
|Y_I(n)|^q &  \le     2^q\Big(  \sum_{L=0}^\infty v_L(n)  R_L\Big)^q  + 2^q \Big(\sum_{L=1} ^\infty w_L(n)  Q_L\Big)^q
\\
&=   2^q  \sum_{L=0}^\infty v_L(n)^q  R_L^q  + 2^q \sum_{L=1} ^\infty w_L(n)^q Q_L^q.
\end{align*}
It follows that
\begin{align*}
	\|Y_I\|_{\ell^q}^q & =   \sum_{n=1}^\infty |Y_I(n)|^q  \le      \sum_{n = 1}^\infty   \Big( 2^q \sum_{L=0}^\infty v_L(n)^q  R_L^q  + 2^q \sum_{L=1}^\infty w_L(n)^q Q_L^q \Big)\\
	& =    2^q \sum_{L=0}^\infty   \|v_L\|_{\ell^q}^q  R_L^q  +  2^q \sum_{L=1}^\infty    \|w_L\|_{\ell^q}^q Q_L^q .
\end{align*}
Thus, by applying the inequality \eqref{xy-small}, we  get
\begin{align*}
\|Y_I\|_{\ell^q}^p & \le   \Big \{   2^q \sum_{L=0}^\infty   \|v_L\|_{\ell^q}^q  R_L^q  +  2^q \sum_{L=1}^\infty    \|w_L\|_{\ell^q}^q Q_L^q  \Big\}^{p/q} 
\\
& \le   2^p     \sum_{L=0}^\infty \|v_L\|_{\ell^q}^p  R_L^p  + 2^p  \sum_{L=1}^\infty \|w_L\|_{\ell^q}^p Q_L^p 
\end{align*}
and thus 
\begin{align}\label{E-Y-I}
	\E[\|Y_I\|_{\ell^q}^p]  \le 2^p    \sum_{L=0}^\infty \|v_L\|_{\ell^q}^p \cdot \E[R_L^p]  + 2^p \sum_{L=1}^\infty \|w_L\|_{\ell^q}^p\cdot \E[ Q_L^p].
\end{align}

\medskip
{\flushleft \it Step 7.  Simple estimates of the quantities $\|v_L\|_{\ell^q}^p$ and $\|w_L\|_{\ell^q}^p$.}
\medskip

By the definitions \eqref{def-v-0}, \eqref{def-v-l} and \eqref{def-wl} for $v_L$ and $w_L$, we have 
\begin{align}\label{cal-v-0}
	\|v_0\|_{\ell^q}^p   = \Big( \sum_{n=1}^{2^k} n^{\frac{\tau q}{2}}\Big)^{p/q}  \lesssim     2^{(\frac{\tau p}{2} +\frac{p}{q})k}
\end{align}
and for all $L\ge 1$, 
\begin{align}\label{vL-wL-1}
	\|v_L\|_{\ell^q}^p  = \Big(  \sum_{2^{L+k-1}<n\le 2^{L+k}} n^{\frac{\tau q}{2}}  \Big)^{p/q}  \lesssim       2^{(\frac{\tau p}{2} +\frac{p}{q})(L+k)},  
\end{align}
\begin{align}\label{vL-wL}
	\|w_L\|_{\ell^q}^p    = \Big(  \sum_{2^{L+k-1}<n\le 2^{L+k}} n^{\frac{\tau q}{2}-q }  \Big)^{p/q}  \lesssim   2^{ (\frac{\tau p}{2} - p+\frac{p}{q})(L+k)}.
\end{align}

\medskip
{\flushleft \it Step 8. Estimate of $\E[R_0^p]$.}
\medskip

Recall the definition \eqref{def-R-0} for $R_0$. By the triangle inequality, we have 
\[
(\E[R_0^p])^{1/p} = \Big\| 
\int_{I}\Big[\prod_{j=0}^{k-1}X_j(t)\Big]  |  \mathring{X}_{k} (t)| \mathrm{d}t\Big\|_{L^p(\PP)}
\le 
\int_{I}  \Big\|  \Big[\prod_{j=0}^{k-1}X_j(t)\Big]  |  \mathring{X}_{k} (t)|  \Big\|_{L^p(\PP)} \mathrm{d}t.
\]
Since the stochastic processes $\{X_j\}_{0\le j \le k}$ are independent, by Lemma~\ref{p-moment-X_m}, we have 
\begin{align}\label{Dt-Lp}
	\Big\|  \Big[\prod_{j=0}^{k-1}X_j(t)\Big]  |  \mathring{X}_{k} (t)|  \Big\|_{L^p(\PP)} =   \Big(\Big[\prod_{j=0}^{k-1} \E[X_j^p(t)]\Big]   \cdot \E [|  \mathring{X}_{k} (t)|^p]\Big)^{1/p} 
	\lesssim   2^{\frac{(p-1)\gamma^2}{2}k}. 
\end{align}
Therefore,  by recalling $|I|=2^{-(k-1)}$, we obtain 
\begin{align}\label{R0-es}
	(\E[R_0^p])^{1/p} \lesssim  2^{[\frac{(p-1)\gamma^2}{2}-1]k}. 
\end{align}

\medskip
{\flushleft \it Step 9. Control of the difference $D_k(t)-D_k(s)$.}
\medskip

From the expressions \eqref{def-Rl} and \eqref{def-Ql}, we are led to study the difference $D_k(t)-D_k(s)$.  Then the elementary identity \eqref{elem-id} will be used again.  To ease the notation, we rewrite $D_k$ introduced in \eqref{def-Dk} as 
\[
D_k(t)= \prod_{j=0}^{k} Z_j(t)  \quad \text{with}\quad Z_j(t)=
\left\{ \begin{array}{cl}
	X_j(t) & \text{if $0\le j \le k-1$}
	\vspace{2mm}
	\\
	\mathring{X}_k(t) & \text{if $j=k$} 
\end{array}
\right..
\]
Note that since $\mathring{X}_k (t) =X_k(t) -1$, we have  
\[
\mathring{X}_k(t) - \mathring{X}_k(s) = X_k(t)-X_k(s)
\]
and thus 
\[
|Z_j(t)-Z_j(s)| = |X_j(t)-X_j(s)| \quad \text{for all $0\le j \le k$}.
\]
Then by \eqref{elem-id}, we obtain 
\begin{align*}
		|D_k(t)-D_k(s)|  &\le   \sum_{r=0}^{k}\Big(\prod_{j=0}^{r-1} |Z_j(s)|\Big) \big|Z_r(t)-Z_r(s)\big| \Big(\prod_{j=r+1}^{k}|Z_j(t)|\Big)
		\\
		& =  \sum_{r=0}^{k}\Big(\prod_{j=0}^{r-1} X_j(s) \Big) \big|X_r(t)-X_r(s)\big| \Big(\prod_{j=r+1}^{k}|Z_j(t)|\Big).
\end{align*}
Therefore, for any $t, s\in [0,1]$ such that $|t-s|\le 2^{-k}$, by applying  the triangle inequality, the independence of $\{X_j\}_{0\le j \le k}$ and then Lemma~\ref{p-moment-X_m},  Lemma~\ref{Holder-p-moment-X_m}, we obtain 
\begin{align}\label{Dt-Ds-E}
	\begin{split}
		\| D_k(t)-D_k(s)\|_{L^p(\PP)} &  \le    \sum_{r=0}^{k}\Big\|\Big(\prod_{j=0}^{r-1} X_j(s) \Big) \big|X_r(t)-X_r(s)\big| \Big(\prod_{j=r+1}^{k}|Z_j(t)|\Big)\Big\|_{L^p(\PP)}
		\\
		& \lesssim \sum_{r=0}^k   2^{\frac{(p-1)\gamma^2}{2}k} \sqrt{2^r|t-s|} 
		\\
		& \lesssim 2^{\frac{(p-1)\gamma^2 +1}{2}k} \sqrt{|t-s|} . 
	\end{split}
\end{align}

\medskip
{\flushleft \it Step 10. Estimate of $\E[R_L^p]$ for $L\ge 1$.}
\medskip

Recall the definition \eqref{def-Rl} of $R_L$ for $L\ge 1$. We have 
\begin{align*}
	\|R_L\|_{L^p(\PP)}&=  \Big\|  \sum_{J\in \mathscr{D}_{L+k-1}(I)}  \int_{J} |D_{k} (t) -D_k(\ell_J) | \mathrm{d}t\Big\|_{L^p(\PP)}
	\\
	&  \le  \sum_{J\in \mathscr{D}_{L+k-1}(I)}  \int_{J}   \|  D_{k} (t) -D_k(\ell_J)  \|_{L^p(\PP)} \mathrm{d}t.
\end{align*}
Now by \eqref{Dt-Ds-E} and the fact that $|t-\ell_J|\le |J| = 2^{-(L+k-1)}\leq2^{-k}$ for all $t\in J\in \mathscr{D}_{L+k-1}(I)$, we obtain 
\[
(\E[R_L^p])^{1/p} \lesssim   \sum_{J\in \mathscr{D}_{L+k-1}(I)}     2^{\frac{(p-1)\gamma^2 +1}{2}k}  \cdot 2^{-\frac{3}{2}(L+k)}.
\]
Note that by the definition \eqref{DLk-1I} of $\mathscr{D}_{L+k-1}(I)$ (recall that the interval $I$ is divided into $2^L$ equal pieces), we have 
\[
\# \mathscr{D}_{L+k-1}(I)  = 2^L. 
\]
Hence we get
\begin{align}\label{RL-es}
	(\E[R_L^p])^{1/p}   \lesssim 2^{L} \cdot 2^{\frac{(p-1)\gamma^2 +1}{2}k}   \cdot 2^{-\frac{3}{2}(L+k)}  = 2^{-\frac{L}{2}} \cdot 2^{[\frac{(p-1)\gamma^2}{2}-1]k}. 
\end{align}

\medskip
{\flushleft \it Step 11. Estimate of $\E[Q_L^p]$ for $L\ge 1$.}
\medskip

Recall the definition \eqref{def-Ql} of $Q_L$ for $L\ge 1$. We have  
\begin{align*}
	(\E[Q_L^p])^{1/p} & = \frac{1}{2\pi}\Big\|
	|D_k(t_{2^L-1})|  + |D_k(t_0)| + \sum_{l=1}^{2^L-1} | D_k(t_{l-1}) - D_k(t_l)|\Big\|_{L^p(\PP)}
	\\
	& \lesssim \| D_k(t_{2^L-1}) \|_{L^p(\PP)}   + \|D_k(t_0)\|_{L^p(\PP)}  + \sum_{l=1}^{2^L-1} \| D_k(t_{l-1}) - D_k(t_l)\|_{L^p(\PP)}. 
\end{align*}
Note that by the same calculation as in \eqref{Dt-Lp} (or by directly using the translation-invariance), we have  
\[
\| D_k(t_{2^L-1}) \|_{L^p(\PP)}   =  \|D_k(t_0)\|_{L^p(\PP)}  \lesssim 2^{\frac{(p-1)\gamma^2}{2}k}. 
\]
On the other hand, by \eqref{Dt-Ds-E} and by using $|t_l-t_{l-1}| = 2^{-(L+k-1)}\leq2^{-k}$, we obtain 
\[
\| D_k(t_{l-1}) - D_k(t_l)\|_{L^p(\PP)} \lesssim  2^{\frac{(p-1)\gamma^2 +1}{2}k} \cdot 2^{-\frac{1}{2}(L+k)}.
\]
It follows that 
\begin{align}\label{QL-es}
	(\E[Q_L^p])^{1/p} \lesssim  2^{\frac{(p-1)\gamma^2}{2}k} + 2^L \cdot 2^{\frac{(p-1)\gamma^2 +1}{2}k} \cdot 2^{-\frac{1}{2}(L+k)} \lesssim 2^{\frac{L}{2}} \cdot 2^{\frac{(p-1)\gamma^2 }{2}k} .
\end{align}
\begin{remark*}
	Note that here $(\E[Q_L^p])^{1/p}$ is large when $k$ or $L$ is large.  However,  the product 
	$
	\|w_L\|_{\ell^q}^p \cdot  \E[Q_L^p]
	$
	becomes very small. 
\end{remark*}

\medskip
{\flushleft \it Step 12. Conclusion of the estimate of $\E[\|Y_I\|_{\ell^q}^p]$.}
\medskip

Combining the inequalities \eqref{cal-v-0} and \eqref{R0-es}, we obtain 
\[
\|v_0\|_{\ell^q}^p \cdot \E[R_0^p]       \lesssim     2^{(\frac{\tau p}{2} +\frac{p}{q})k} \cdot (2^{[\frac{(p-1)\gamma^2}{2}-1]k})^p  = 2^{-k (p - \frac{p(p-1)\gamma^2}{2} - \frac{\tau p}{2} - \frac{p}{q})}.
\]
For all integers $L\ge 1$, by   \eqref{vL-wL-1}, \eqref{RL-es}, 
\begin{align*}
\|v_L\|_{\ell^q}^p \cdot \E[R_L^p] & \lesssim  2^{(\frac{\tau p}{2} +\frac{p}{q})(L+k)}\cdot  (2^{-\frac{L}{2}} \cdot 2^{[\frac{(p-1)\gamma^2 }{2}-1]k})^p 
\\
& = 2^{-k (p - \frac{p(p-1)\gamma^2}{2} - \frac{\tau p}{2} - \frac{p}{q})} \cdot 2^{-pL(\frac{1-\tau}{2} -\frac{1}{q})}
\end{align*}
and by   \eqref{vL-wL},  \eqref{QL-es}, 
 \begin{align*}
 \|w_L\|_{\ell^q}^p \cdot \E[Q_L^p] & \lesssim  2^{ (\frac{\tau p}{2} - p+\frac{p}{q})(L+k)}\cdot  (2^{\frac{L}{2}} \cdot 2^{\frac{(p-1)\gamma^2 }{2}k})^p
 \\
 &  = 2^{-k (p - \frac{p(p-1)\gamma^2}{2} - \frac{\tau p}{2} - \frac{p}{q})} \cdot 2^{-pL(\frac{1-\tau}{2} -\frac{1}{q})}. 
\end{align*}
Therefore, by \eqref{E-Y-I}, we obtain 
\[
	\E[\|Y_I\|_{\ell^q}^p]  \lesssim  2^{-k (p - \frac{p(p-1)\gamma^2}{2} - \frac{\tau p}{2} - \frac{p}{q})} \Big[1 + \sum_{L=1}^\infty  2^{-pL(\frac{1-\tau}{2} -\frac{1}{q})} +  \sum_{L=1}^\infty  2^{-pL (\frac{1-\tau }{2} - \frac{1}{q})} \Big]. 
\]
Since $q>\frac{4}{1-\tau}$, we have 
\[
\frac{1-\tau}{2}- \frac{1}{q}>  \frac{1-\tau}{2}-  \frac{1-\tau}{4} =   \frac{1-\tau}{4}>0.
\]
Hence we get
\[
\sum_{L=1}^\infty  2^{-pL (\frac{1-\tau }{2} - \frac{1}{q})} <\infty. 
\]
Consequently, we get the desired inequality 
\[
\E[\|Y_I\|_{\ell^q}^p]  \lesssim 2^{-k (p - \frac{p(p-1)\gamma^2}{2} - \frac{\tau p}{2} - \frac{p}{q})} .
\]
This completes the whole proof of Proposition~\ref{UB-ZW}.

\section{Proof of Theorem~\ref{FourierDim-GMC}}

Theorem~\ref{FourierDim-GMC} follows immediately from Lemma~\ref{Lower-Bound-FourierDim-GMC} and Lemma~\ref{Upper-Bound-FourierDim-GMC} below.  

\subsection{Lower bound of Fourier dimension for GMC}
\begin{lemma}\label{Lower-Bound-FourierDim-GMC}
	For each $\gamma\in(0,\sqrt{2})$,  almost surely,  we have $\mathrm{dim}_{F}(\mu_{\gamma,\mathrm{GMC}})\geq D_{\gamma}$. 
\end{lemma}

\begin{proof}
	Fix any $\gamma\in(0,\sqrt{2})$. For any $\tau\in(0,D_{\gamma})$, take $p=p(\gamma,\tau)$ and $q=q(\gamma,\tau)$ as in Theorem~\ref{Uniform-Bound}. Then by \eqref{weacontoGMC}, Theorem~\ref{Uniform-Bound} combined with the standard fact for vector-valued martingales implies that 
	\[
	\mathbb{E}\Big[\Big\{\sum_{n=1}^{\infty}n^{\frac{\tau q}{2}}\big|\widehat{\mu_{\gamma,\mathrm{GMC}}}(n) \big|^{q}\Big\}^{p/q}\Big] = \sup_{m\geq 1}\mathbb{E}[\|M_m\|_{\ell^{q}}^{p}]<\infty
	\]
	and hence 
	\[
	\sum_{n=1}^{\infty}n^{\frac{\tau q}{2}}\big|\widehat{\mu_{\gamma,\mathrm{GMC}}}(n) \big|^{q}<\infty \quad a.s. 
	\]
	Consequently,  almost surely, 
	\[
	|\widehat{\mu_{\gamma,\mathrm{GMC}}}(n)|^2=O(n^{-\tau})  \quad \text{as $n\to\infty$},
	\]
	and hence $\mathrm{dim}_{F}(\mu_{\gamma,\mathrm{GMC}})\ge \tau$.  Finally, by taking a sequence $\{\tau_N\}_{N\geq1}\subset(0,D_{\gamma})$ with $\lim_{N\to\infty}\tau_N=D_{\gamma}$, we conclude that,  almost surely, 
	$\mathrm{dim}_{F}(\mu_{\gamma,\mathrm{GMC}})\geq D_{\gamma}$. 
\end{proof}

\subsection{Upper bound of Fourier dimension for GMC}

\begin{lemma}\label{Upper-Bound-FourierDim-GMC}
	For each $\gamma\in(0,\sqrt{2})$,  almost surely,  we have $\mathrm{dim}_{F}(\mu_{\gamma,\mathrm{GMC}})\leq D_{\gamma}$. 
\end{lemma}

\subsubsection{Proof of Lemma \ref{Upper-Bound-FourierDim-GMC} via $L^2$-spectrum of GMC}

In   \cite{Ber23}, Bertacco  introduced the $L^q$-spectrum of a measure $\nu$ by 
\[
\tau_\nu(q): =   \limsup_{\delta\to 0^+}\frac{\displaystyle \log \Big(\sup \sum_i \nu \big(B(x_i, \delta)\big)^q \Big)}{-\log \delta}, \quad q\in \R,
\]
where the supremum is taken over all families of disjoint balls. From Bertacco's definition, we have  (the $\limsup$ becomes $\liminf$ after multiplication by $-1$)
\begin{align}\label{B-Lq-def}
-\tau_\nu(q) =   \liminf_{\delta \to 0^+}\frac{\displaystyle \log \Big(\sup \sum_i \nu \big(B(x_i, \delta)\big)^q \Big)}{\log \delta}.
\end{align}
By comparing  \eqref{def-sup-ball} and \eqref{B-Lq-def}, we get 
\[
\dim_2(\nu)= - \tau_\nu(2). 
\]
Note that the above equality is a particular case of \cite[Lemma 2.6.6]{BSS23}. 

\begin{remark*}
One may note that the right hand side of \eqref{B-Lq-def} is the definition of the $L^q$-spectrum of $\nu$ in \cite[Definition 2.6.7 and  Formula (2.30)]{BSS23}. 
\end{remark*}

Bertacco \cite[Theorem 3.1 and Formula (3.2)]{Ber23}, as well as Rhodes-Vargas  \cite[Section~4.2]{RV14} and   Garban-Vargas \cite[Remark~2]{GV23},  computed the $L^q$-spectrum  $\tau_{\mu_{\gamma,\mathrm{GMC}}}(q)$ of the sub-critical GMC measures  for all $q\in \R$ and all dimensions $d\ge 1$, where  the more general perturbed log-kernel of the form \eqref{bdd-pert} was studied.   In particular, for our purpose of the GMC measure with dimension $d=1$, it was shown that for any $\gamma \in(0,\sqrt{2})$,  almost surely,  
\begin{align*}
\dim_2(\mu_{\gamma,\mathrm{GMC}})= - \tau_{\mu_{\gamma,\mathrm{GMC}}}(2)= 
\left\{
\begin{array}{cc}
\xi_{\mu_{\gamma,\mathrm{GMC}}}(2) - 1 & \text{if $2\le \sqrt{2}/\gamma$}
\vspace{2mm}
\\
2 \xi_{\mu_{\gamma,\mathrm{GMC}}}'(\sqrt{2}/\gamma) & \text{if $2\ge  \sqrt{2}/\gamma$} 
\end{array}
\right., 
\end{align*}
where $\xi_{\mu_{\gamma,\mathrm{GMC}}}(q)$ (see \cite[Formula~(2.5)]{Ber23}) is the power law spectrum of the GMC measure $\mu_{\gamma,\mathrm{GMC}}$ given by 
\[
\xi_{\mu_{\gamma,\mathrm{GMC}}}(q) = \Big(1 + \frac{1}{2}\gamma^2 \Big)q - \frac{1}{2}\gamma^2 q^2, \quad q \in \R. 
\]
Therefore, by an elementary computation,  for any $\gamma \in(0,\sqrt{2})$,  almost surely we have,  
\[
\dim_2(\mu_{\gamma,\mathrm{GMC}}) = \left\{
\begin{array}{cl}
1-\gamma^2 & \text{if $0<\gamma \le \sqrt{2}/2$}
\vspace{2mm}
\\
2+\gamma^2 - 2 \sqrt{2} \gamma & \text{if $\sqrt{2}/2\le \gamma <\sqrt{2}$} 
\end{array}
\right..
\]
In other words,  for any $\gamma \in(0,\sqrt{2})$,  recall $D_\gamma$ given by \eqref{D-gamma}, then almost surely,  
\[
\dim_2(\mu_{\gamma,\mathrm{GMC}})   = D_\gamma.
\]

Finally, we complete the proof of Lemma \ref{Upper-Bound-FourierDim-GMC} by applying the standard inequality \eqref{F-less-cor} in potential theory: $\dim_F(\nu)\le \dim_2(\nu)$ for any finite Borel measure $\nu$ supported on a compact subset.

\subsubsection{An alternative proof of Lemma \ref{Upper-Bound-FourierDim-GMC}}
The almost sure  upper bound  
\begin{align}\label{up-bdd-pf}
\mathrm{dim}_{F}(\mu_{\gamma,\mathrm{GMC}})\leq D_{\gamma} \quad a.s. 
\end{align}
 is also known to Lacoin-Rhodes-Vargas and Garban-Vargas, see \cite{LRV15}  and also \cite[Remark~2]{GV23}.   Here, based on \cite[Remark~2]{GV23}, we provide an alternative proof of \eqref{up-bdd-pf}.

For proving \eqref{up-bdd-pf}, we need to use the result in \cite{LRV15} and a simple application of the Kolmogorov's zero-one law.   Let us fix any $\gamma\in(0,\sqrt{2})$.  For any $\beta>0$, consider the  event $A_\beta$ defined by 
\[
A_\beta:= \left\{ \int_{[0,1]^2} \frac{\mu_{\gamma,\mathrm{GMC}}(\mathrm{d}t) \mu_{\gamma,\mathrm{GMC}}(\mathrm{d}s)}{|t - s|^\beta}  < \infty \right\}.
\]

\begin{lemma}\label{0-1-law}
	For any $\beta>0$, we have $\PP(A_\beta)\in \{0,1\}$. 
\end{lemma}
\begin{proof}
	Recall the definitions  \eqref{def-psim}, \eqref{def-phim} for the stochastic processes $\psi_m$ and $\varphi_m$ respectively. 
	For any $m\geq 1$, set 
	\begin{align*}
		\mu_{\gamma, >m}(\mathrm{d}t) :=& \prod_{k=m+1}^\infty \exp\Big(\gamma \varphi_k(t) -  \frac{\gamma^2 \mathbb{E} [\varphi^2_k(t)]}{2}\Big) \mathrm{d}t 
		\\
		=& \lim_{N\to\infty} \prod_{k=m+1}^N \exp\Big(\gamma \varphi_k(t) -  \frac{\gamma^2 \mathbb{E} [\varphi^2_k(t)]}{2}\Big) \mathrm{d}t.
	\end{align*}
	Similar to the random measure $\mu_{\gamma, \GMC}(\mathrm{d}t) $, the existence of $ \mu_{\gamma, >m}(\mathrm{d}t) $ is also guaranteed by \cite{Kah85a}.

	Clearly, we have 
	\[
	\mu_{\gamma,\mathrm{GMC}}(\mathrm{d}t) = \underbrace{\exp\Big(\gamma \psi_m(t) - \frac{\gamma^2}{2}(m\log 2+1)\Big)}_{\text{denoted $R_{\gamma, m}(t)$}} \mu^\gamma_{>m}(\mathrm{d}t). 
	\]
	Therefore, 
	\[
	\int_{[0, 1]^2} \frac{\mu_{\gamma,\mathrm{GMC}}(\mathrm{d}t)\mu_{\gamma,\mathrm{GMC}}(\mathrm{d}s)}{|t - s|^\beta} =\int_{[0, 1]^2} \frac{R_{\gamma, m}(t) R_{\gamma, m}(s) }{|t - s|^\beta}\mu_{\gamma, >m}(\mathrm{d}t)\mu_{\gamma, >m}(\mathrm{d}s).
	\]
	By Corollary~\ref{cor-holder} and Convention~\ref{conv-mod},  almost surely, $R_{\gamma,m}(t)$ is continuous on $t$ and non-vanishing. Thus, by setting 
	\[
	B_\beta(m) := \left\{\int_{[0,1]^2}\frac{\mu_{\gamma, >m}(\mathrm{d}t)\mu_{\gamma, >m}(\mathrm{d}s)}{|t-s|^\beta} < \infty\right\} \in \sigma (\varphi_{k}: k>m), 
	\]
	we have $A_\beta = B_\beta(m)$ up to a probability measure zero set.  Since $m$ is arbitrary,  the result follows by applying Kolmogorov's zero-one law.
\end{proof}

\begin{proof}[Proof of the upper bound \eqref{up-bdd-pf}]
	Assume by contradiction that  $\dim_F(\mu_{\gamma, \GMC}) > D_\gamma$ with positive probability.
	That is, there exists an $\varepsilon>0$ with $D_\gamma+\varepsilon < 1$ such that 
	\begin{align}\label{pos-p}
		|\widehat{\mu_{\gamma,\GMC}}(\xi)|^2 = O (|\xi|^{-(D_\gamma+2 \varepsilon)})  \quad  \text{with positive probability}.
	\end{align}
	By the standard equality for the Riesz energy (see, e.g., \cite[Theorem~3.10]{Mat15}), 
	\[
	\int_{[0, 1]^2} \frac{\mu_{\gamma, \GMC}(\mathrm{d}t)\mu_{\gamma, \GMC} (\mathrm{d}s)}{|t - s|^{D_\gamma+\varepsilon}} =  \pi^{D_\gamma+\varepsilon-1/2} \frac{\Gamma\big(\frac{1-D_\gamma-\varepsilon}{2}\big)}{\Gamma\big(\frac{D_\gamma+\varepsilon}{2}\big)} \int_\mathbb{R} |\widehat{\mu_{\gamma, \GMC}}(\xi)|^2|\xi|^{D_\gamma+\varepsilon-1}\mathrm{d}\xi.
	\] 
	Therefore,   by \eqref{pos-p}, 
	\[
	\int_{[0, 1]^2} \frac{\mu_{\gamma, \GMC}(\mathrm{d}t)\mu_{\gamma, \GMC} (\mathrm{d}s)}{|t - s|^{D_\gamma+\varepsilon}}<\infty \quad  \text{with positive probability}.
	\]
	Hence,  by Lemma~\ref{0-1-law}, 
	\[
	\int_{[0, 1]^2} \frac{\mu_{\gamma, \GMC}(\mathrm{d}t)\mu_{\gamma, \GMC} (\mathrm{d}s)}{|t - s|^{D_\gamma+\varepsilon}}<\infty \quad a.s. 
	\]
	However, this contradicts to the following  result from \cite{LRV15} (see also \cite[Remark~2]{GV23}): 
	\[
	\int_{[0, 1]^2} \frac{\mu_{\gamma, \GMC}(\mathrm{d}t)\mu_{\gamma, \GMC} (\mathrm{d}s)}{|t - s|^{\beta}}<\infty \quad a.s. \quad \text{if and only if}\quad \beta<D_\gamma.
	\]
	This completes the proof of the almost sure upper bound \eqref{up-bdd-pf}.
\end{proof}


\newcommand{\etalchar}[1]{$^{#1}$}

\end{document}